\documentclass[preprint]{siamonline171218}

\usepackage{amssymb}
\usepackage{amsmath}
\usepackage{subfigure}
\usepackage{lipsum}
\usepackage{amsfonts}
\usepackage{graphicx}
\usepackage{epstopdf}
\usepackage{algorithmic}
\usepackage{hyperref}

\newsiamremark{remark}{Remark}

\headers{Delayed Interactions in Active Agents: Stability and Formations}{Y. Wang, A. Amann, J. Cao, J. Kurths, and S. Yanchuk}

\title{Delayed Interactions in Active Agents: Stability and Formations \thanks{Submitted to the editors DATE.
\funding{This publication has emanated from research jointly funded by Taighde Éireann – Research Ireland under Grant number FFP-A/12066.}}}

\author{Yu Wang%
  \footnotemark[6]
  \thanks{Institute of Mathematics, Humboldt-Universität zu Berlin, 10099 Berlin, Germany.}
  \and Andreas Amann%
  \footnotemark[6]
  \thanks{School of Mathematical Sciences, University College Cork, Cork T12 XF62, Ireland.   (\email{syanchuk@ucc.ie})}
  \and Jinde Cao%
  \thanks{School of Mathematics, Southeast University, 210096 Nanjing, China.}
  \and Jürgen Kurths%
  \footnotemark[6] 
  \thanks{Department of Physics, Humboldt-Universität zu Berlin, 10099 Berlin, Germany.}
  \and Serhiy Yanchuk%
  \footnotemark[3]
 \thanks{Potsdam Institute for Climate Impact Research, 14473 Potsdam, Germany.}
}

\begin{document}

\maketitle

\begin{abstract}
Active agents with time-delayed interactions arise naturally in various real-world systems, such as biological systems, transportation networks and robotic swarms.
Such systems are typically modeled as Delay Differential Equations (DDEs) that incorporate inertial effects.
In this paper, we investigate the stability of pattern formation of active agents with inertia and time delays, in both uncoupled and coupled scenarios.
We derive and analyze a high-dimensional linear DDE model that characterizes the stability of such formations.
Starting with the uncoupled scenario, where agents are driven only by a virtual leader, we describe the stability spectrum and provide conditions for the delay-independent (absolute) stability of the formations, as well as delay-dependent stability and unstable hyperbolic behavior. 
Different cases correspond to distinct universality classes of the corresponding spectrum.
For the coupled scenario, where agents are driven by both the virtual leader and inter-agent interactions, we consider both symmetric and non-symmetric coupling topologies.
Here we also provide an explicit spectrum classification, including the absolute stability criterion.
Additionally, we investigate interactions in the large-delay limit, where delays affect inter-agent coupling, while local feedback remains instantaneous.
In this limit, we prove rigorously that the stability region in the complex plane of the eigenvalues of the Laplacian matrix converges to a circle centered at the origin, a phenomenon previously observed in delay-coupled networks.
Our findings provide a universal framework for understanding stable formations and motions of active agents with delayed interactions.
\end{abstract}

\section{Introduction}
\label{sec:Intro}




Active agents possess the capability to move autonomously by perceiving their environment and information about interactions accordingly. 
They often exhibit collective behaviors when such agents interact with each other~\cite{shoham2008multiagent, sar2023flocking}.
Such systems are prevalent across many fields, ranging from bacterial colonies~\cite{shklarsh2011smart, sar2023flocking} and bird flocks~\cite{olfati2006flocking} to robotic swarms~\cite{ota2006multi,scholz2018inertial,lopez2019adaptive,paramanick2025spontaneous}. 
The processes of perception, feedback, and decision-making of information in these systems inherently involve time delays, resulting in delayed active agent systems.
Time delays significantly alter system dynamics, including behaviors such as oscillatory or chaotic motion, and complex behavior formation.
Active agents with time delays can be found in various applications, cf.~Refs.~\cite{Hale1993,xiao2008consensus,leonard2014multi,koh2016consensus,saberi2022synchronization,wang2024designing,wang2025intelligent,yao2025comal,Erneux2017} for a recent overview.  

In systems of active agents with time-delayed interactions, each agent adjusts its motion based on delayed information about its own state or the states of connected agents.
These interactions are commonly modeled using Delay Differential Equations (DDEs) with inertial effects, which emphasize the influence of past states on current dynamics~\cite{erneux2009applied,hindes2016hybrid,koh2016consensus,glass2021nonlinear,yao2025comal,yanchuk2017spatio,Soriano2013,krauskopf_bifurcation_2023}.
The inertial delay observed in active agents arises from the finite relaxation time in their motion, which deeply influences long-term dynamics and therefore typically requires consideration~\cite{scholz2018inertial,caprini2024dynamical}.
In this study, we focus on a class of models in which the dynamics of agents are governed by a high-dimensional linear system of DDEs, incorporating time delays into both local feedback and interactions. 
A key property of such systems is that the stability of collective formations can be explained by the stability of the linear high-dimensional DDE system.
By analyzing the spectrum of this DDE system and in particular by employing the Asymptotic Continuous Spectrum (ACS) technique \cite{Wolfrum2006,Yanchuk2015b,ruschel2021spectrum,Yanchuk2022b,wang2024universal}, we can determine the conditions whether a desired formation will remain stable, persists, or become unstable in the presence of time delays.

For a general linear time-delayed system, absolute stability guarantees stability for all delay values.
In Ref.~\cite{Yanchuk2022b}, we derived explicit criteria for the absolute stability in linear DDEs. 
Further development of these results in Ref.~\cite{wang2024universal} provides a universality classification of linear DDEs such that each class has the same sequence of stabilizing or destabilizing bifurcations when time delay is changing. 
Thus ACS is a useful tool for analyzing the stability and bifurcation behavior of DDEs. 
Originally introduced for DDEs with large delays~\cite{Giacomelli1996,Wolfrum2006, yanchuk2010multiple, Lichtner2011, Sieber2013, Yanchuk2015b, yanchuk2017spatio, Klinshov2017, ruschel2021spectrum, Yanchuk2022b, poignard2022self, wang2024universal}, the ACS provides a computationally rather basic framework for characterizing desctabilization scenarios.

In this study, we examine thoroughly two cases: a system of uncoupled agents driven only by a virtual leader (prescribed target function), and a coupled system of active agents not only driven by a virtual leader but also interacting with each other via a time delay.
We employ the ACS and the Master Stability Function (MSF) \cite{pecora1998master,Huddy2020,BOE20,ruschel_master_2025,Hart2015,flunkert2010synchronizing} approaches to provide theoretical analysis of these two cases.
Using the universal classification introduced in \cite{wang2024universal}, we provide explicit conditions for the agent formation to be: (i) stable for all delays (absolutely stable or universality class 0); (ii) to undergo an explicitly determined destabilization sequence as time delay increases (universality class I); (iii) to undergo explicitly determined destabilization and stabilization sequences as time delay increases (universality class II);  or (iv) unstable and hyperbolic for all time delays (universality class $U$). 
In particular, when delayed coupling is introduced, the stability of the formation can be described by a combination of the above cases.

This work is organized as follows.
In Section \ref{sec:Model_N}, we derive a high-dimensional linear DDE model that describes the dynamics of active agents. 
The stability of this model directly determines the stability of the resulting formations.
Section \ref{sec:Preliminaries} describes two analytical tools for studying formations of interacting agents: the master stability function and the asymptotic continuous spectrum.
Section \ref{sec:Uncoupled_Case} focuses on the uncoupled case, where agents are driven solely by the virtual leader. 
We identify four universality classes (0, I, II and U) of ACS of the high-dimensional DDE system and illustrate them with a bifurcation diagram. We then analyze the motion and pattern formation of active agents within each class.
Section \ref{sec:Coupled_Case} extends this analysis to the coupled case, where active agents are driven by a virtual leader and interact with others incorporating time delays. 
We then examine the motion and formation behaviors of the agents under both symmetric and asymmetric interaction topologies.
We show how the spectrum of the whole coupled system can be expressed as a combination of the spectral branches of types 0, I, II, and U.
Furthermore, we examine the large delay limit, in which delays are present only in the interactions.
Notably, in this limit, the stability region in the plane of the complex eigenvalues of the Laplacian (describing the coupling topology) asymptotically approaches a circle centered at the origin. 
This circular property of the ``master stability function'' has previously been observed in delay-coupled networks \cite{flunkert2010synchronizing}, but here we provide a first rigorous proof of this property.
Finally, Section \ref{sec:Conclusions} summarizes our main results and outlines potential directions for future research.

\section{The linear model of active agents}
\label{sec:Model_N}
We consider the following model of $N$ active agents 
\cite{oh2015survey,suzuki2016leader,han2017multi,scholz2018inertial, chen2019trajectory,li2021fuzzy, zhang2023optimal,wang2023multi, goh2023hydrodynamic,  xiao2024perception,guo2025distributed} with inertia and time-delayed control
\begin{align} 
\label{eq:Motion_AP_a}
\mathbf{\dot R}(t)  
        &=\mathbf{V}(t),
\\ \label{eq:Motion_AP_b}
\mathbf{\dot V}(t)  
        &=\mathbf{U}\left( \mathbf{R}(t),\mathbf{V}(t),\mathbf{R}(t-\tau),\mathbf{V}(t-\tau),t\right), 
\end{align} 
where $\mathbf{R}: \mathbb{R} \to \mathbb{R}^{3N}$ is the position,  
$\mathbf{R}(t) = \left(R_1(t),\dots,R_N(t)\right)$,  $R_i(t)\in \mathbb{R}^3$, 
$\mathbf{V} : \mathbb{R} \to \mathbb{R}^{3N}$ is the velocity,  
$\mathbf{V}(t) = \left(V_1(t), \dots,V_N(t)\right)$, $V_{i}(t)\in \mathbb{R}^3$ of moving agents. 
The system moves due to the control input force $\mathbf{U} : \left(\mathbb{R}^{3N}\right)^4\times\mathbb{R} \to \mathbb{R}^{3N}$, 
$\mathbf{U} = \left(U_1, \dots,U_N\right)$, $U_{i}\in \mathbb{R}^3$. 

The main goal of the control force $\mathbf{U}$ is to achieve a desired motion and formation of agents. 
More specifically, the target trajectory $R_i(t)$ of agent $i$  is prescribed by a time-dependent function $R_0(t) +s_i $, where  $s_i \in \mathbb{R}^3$ determines the position of agent $i$ in the formation.
Hence, the vector $\mathbf{s}=\left(s_1,\dots,s_N\right)\in \mathbb{R}^{3N}$ determines the desired formation shapes of active agents. 
By denoting 
$\mathbf{R_T} = \left(R_0(t),\dots, R_0(t)\right) + \mathbf{s} \in \mathbb{R}^{3N}$,  the control goal is to find conditions on the control force $\mathbf{U}$ such that  $\mathbf{R}(t)\to \mathbf{R_T}(t)$ asymptotically and exponentially fast with $t\to \infty$. 
We assume that $R_0(t)$ (and hence $\mathbf{R_T}(t)$) is a two times continuously differentiable function.

We define the error variables for the deviations of the positions $\mathbf{e}(t)=\left(e_1(t),\dots,e_N(t)\right)$ and velocities $\boldsymbol{\xi}(t)=\left(\xi_1(t),\dots,\xi_N(t)\right)$ from the target functions as
\begin{align}
\label{Eq:Error_P}
\mathbf{e} 
        & =\mathbf{R} - \mathbf{R_T} =\mathbf{R} - \left(\mathrm{1}_N \otimes R_0(t)+ \mathbf{s}\right),
\\ \label{Eq:Error_V}
\boldsymbol{\xi} 
         & =\mathbf{\dot R} - \mathbf{\dot R_T} = \mathbf{V} - \mathrm{1}_N \otimes \dot R_0(t).
\end{align}
Here $\otimes$ denotes the Kronecker product and $\mathrm{1}_N$ is the $N$-dimensional vector with identical elements equal to 1.

We consider the Proportional Derivative (PD) controller with delayed terms, as studied in Refs.~\cite{tomei1991adaptive,johnson2005pid,hamamci2010calculation, liu2021pd,hernandez2019practical, ma2021delay,wang2023pd}.
Adding delay terms to the standard PD controller can improve its response characteristics and ensures better performance.
The considered PD controller has the following form:
\begin{align}\nonumber
U_{i} =
        &-\sum^N_{j=1,j\neq i}a_{ij}\Bigl\{  k \left[ \left(R_i(t)-R_j(t)\right) -\left(s_i-s_j\right)\right]
\\ \nonumber
        &+ k^\tau \left[\left(R_i(t-\tau)-R_j(t-\tau)\right)-\left(s_i-s_j\right)\right]
\\ \nonumber
         &+h \left[V_i(t)-V_j(t)\right]+ h^\tau \left[V_i(t-\tau)-V_j(t-\tau)\right]\Bigr\}
\\ \nonumber 
         &-k_0\left[R_i(t)-\left(R_0(t)+s_i\right)\right]-k^\tau_0\left[R_i(t-\tau)-\left(R_0(t-\tau)+s_i\right)\right]
\\ \nonumber 
          &-h_0\left[V_i(t)-V_0(t)\right]-h^\tau_0\left[V_i(t-\tau)-V_0(t-\tau)\right]+ U_0(t), 
\\  \label{eq:ControlInput2} 
          &\quad i=1,2,\ldots, N.
\end{align}
The coefficients $a_{ij}$, $j\ne 0$ determine the coupling between the agents $i$ and $j$; $a_{ij}> 0$ if the information is transmitted from $j$ to $i$, otherwise, $a_{ij}=0$.
$k$, $k^\tau$, $h$, $h^\tau$ are the corresponding control gains between the agents.   
$k_0$, $k^\tau_0$, $h_{0}$, and $h^\tau_{0}$  determine the coupling between the agents and the virtual leader. $V_0(t) = \dot R_0(t)$ and $U_0(t)=\ddot R_0(t)$ are the velocity and acceleration of the ``virtual leader'', i.e. a prescribed position in space that should be followed by the active agents.

Using the error variables \eqref{Eq:Error_P} and \eqref{Eq:Error_V},  the considered PD controller reads
\begin{align}\nonumber
U_{i} =
        &-\sum^N_{j=1,j\neq i}a_{ij}\Bigl\{k \left[e_i(t)-e_j(t)\right] +k^\tau \left[e_i(t-\tau)-e_j(t-\tau)\right]
\\ \nonumber&\quad\quad\quad\quad\quad\quad
        +h\left[\xi_i(t)-\xi_j(t)\right]+ h^\tau\left[\xi_i(t-\tau)-\xi_j(t-\tau)\right]\Bigr\} 
\\
        & - k_{0} e_i(t) - k^\tau_{0}e_i(t-\tau)- h_{0}\xi_i(t)- h^\tau_{0}\xi_i(t-\tau)+ U_0(t). 
\label{eq:ControlInput}
\end{align}
In vector form, the controller is expressed as
\begin{align}\nonumber
\mathbf{U} = 
        &-\left( \mathbf{L}\otimes \mathrm{1}_3 \right) \left[k\mathbf{e}(t)+ k^\tau\mathbf{e}(t-\tau)+h\boldsymbol{\xi}(t)+h^\tau\boldsymbol{\xi}(t-\tau)\right]
\\ 
        & - k_{0} \mathbf{e}(t) - k^\tau_{0}\mathbf{e}(t-\tau)- h_{0}\boldsymbol{\xi}(t)- h^\tau_{0}\boldsymbol{\xi}(t-\tau)+ \mathbf{U}_0(t),
\label{eq:controller}
\end{align}
where $\mathbf{L}$ is the Laplacian matrix 
\begin{align}
\label{eq:L}
\left[ \mathbf{L} \right]_{ij}  = 
        \left\{\begin{array}{cc} \sum\limits^N_{l=1,l\ne i} a_{il},  &i=j; 
\\
        - a_{ij},  & i\ne j; \quad i=1,2,\cdots N,
        \end{array}\right. 
\end{align} 
and $\mathbf{U}_0(t) = 1_N\otimes U_0(t)$.

Finally, using the relations $\mathbf{\dot e} = \mathbf{\dot R} - \mathbf{\dot R_T}$ and $\boldsymbol{\dot \xi} = \mathbf{\dot V} - \mathrm{1}_N \otimes \ddot R_0(t) = \mathbf{U} - \mathbf{U}_0(t)$,
we obtain the following autonomous linear delay-differential equation describing the error dynamics
\begin{align}
\mathbf{\dot e} = 
        & \, \boldsymbol{\xi}, 
\label{eq:DDE-e}
\\ 
\boldsymbol{\dot \xi} = 
        & -\left( \mathbf{L}\otimes \mathrm{1}_3 \right)  \left[k\mathbf{e}(t)+ k^\tau\mathbf{e}(t-\tau)+h\boldsymbol{\xi}(t)+h^\tau\boldsymbol{\xi}(t-\tau)\right] \nonumber 
\label{eq:DDE-xi}
\\
        &- k_{0} \mathbf{e}(t) - k^\tau_{0}\mathbf{e}(t-\tau)- h_{0}\boldsymbol{\xi}(t)- h^\tau_{0}\boldsymbol{\xi}(t-\tau).
\end{align}
The system defined by equations \eqref{eq:DDE-e}--\eqref{eq:DDE-xi} governs the stability of the agent formation and the tracking of the virtual leader orbit.
This will be the main object of this study. 

System \eqref{eq:DDE-e}--\eqref{eq:DDE-xi} can be rewritten in the following form
\begin{align}\nonumber
\dot Z(t)=& 
        \left[ \left(\mathbf{1}_N \otimes M -\textbf{L}\otimes P\right) \otimes \mathrm{1}_3\right] Z(t)  \\&+ \left[ \left(\mathbf{1}_N\otimes M^\tau-\textbf{L}\otimes P^\tau\right) \otimes \mathrm{1}_3 \right] Z(t-\tau), 
\label{eq:Gener_MSF}
\end{align}
where $Z(t)=[e_1,\xi_1\cdots,e_{N},\xi_{N}]^T\in \mathbb{R}^{ N\times 2 \times 3}$ is the state vector.
More precisely, $Z$ is a rank three tensor defined by $Z_{i1j} = e_{ij}$ and $Z_{i2j} = \xi_{ij}$ for $i=1,\dots,N$ and $j=1,2,3$. 
Matrices $M,~M^\tau$, $P,~P^\tau\in \mathbb{R}^{2\times 2}$ are given by 
\begin{equation}
\label{eq:MP}
\begin{aligned}
 M=&\left[\begin{array}{cccc}
		0& 1\\
		-k_0& -h_0
	\end{array}\right],~
 M^\tau=\left[\begin{array}{cccc}
		0& 0\\
		-k^\tau_{0}& -h^\tau_{0}
	\end{array}\right],\\
 P=&\left[\begin{array}{cccc}
		0& 0\\
		k& h
	\end{array}\right],~~~~~~~~
 P^\tau=\left[\begin{array}{cccc}
		0& 0\\
		k^\tau& h^\tau
	\end{array}\right].
\end{aligned}
\end{equation}
The connectivity structure is given by the Laplacian matrix $\mathbf{L}\in \mathbb{R}^{N\times N}$ with zero row sum.
For completeness,  equations \eqref{eq:Gener_MSF} can be written using explicit indices and the Einstein sum convention as follows
\begin{align*}
\dot Z_{ijs} =
         (\delta_{im} M_{jl}  - [\mathbf{L}]_{im}P_{jl} )Z_{mls}(t) + (\delta_{im}M^{\tau}_{jl}  - [\mathbf{L}]_{im}P^{\tau}_{jl}) Z_{mls}(t-\tau).
\end{align*}

\section{Preliminaries}
\label{sec:Preliminaries}
This section introduces two tools that will be used for studying the pattern formation of interacting agents: the Master Stability Function (MSF) and the Asymptotic Continuous Spectrum (ACS) approach for delay systems.

\subsection{Master Stability Function for delay-coupled systems \label{sec:MSF}}

We assume that the  Laplacian matrix $\mathbf{L}$ can be diagonalized such that 
\begin{align*}
\mathbf{H}^{-1} \mathbf{L} \mathbf{H} = 
    \boldsymbol{\Lambda} = 
        {\mathrm{diag}}\left\{\lambda_1,\cdots,\lambda_N\right\},
\end{align*}
where $\lambda_{\ell}$ are the eigenvalues of $\mathbf{L}$. 
In this case, system \eqref{eq:Gener_MSF} can be block diagonalized, similarly to the MSF approach in Ref.~\cite{pecora1998master}. 
We define the new variable $X(t)$ using $X(t)=\left(\mathbf{H}\otimes {1_2}\otimes 1_3\right)^{-1}Z(t)$. 
Then, Eq.~\eqref{eq:Gener_MSF} has the block diagonal form with respect to the new variable 
\begin{align}
\nonumber
\dot X(t)= 
        &\left[ \left(\mathbf{1}_N \otimes M -\boldsymbol{\Lambda}\otimes P\right) \otimes \mathrm{1}_3\right] X(t) 
\\
         & + \left[ \left(\mathbf{1}_N\otimes M^\tau-\boldsymbol{\Lambda}\otimes P^\tau\right) \otimes \mathrm{1}_3 \right] X(t-\tau).
\label{eq:Gener_MSF2}
\end{align}
This block structure leads to the set of $N$ independent equations 
\begin{align}
\label{Eq:eq:Gener_MSF}
\dot x_\ell(t)= 
        M x_\ell(t)+M^\tau x_\ell(t-\tau)-\lambda_\ell\left[Px_\ell(t)+P^\tau x_\ell(t-\tau)\right],
\end{align}
where $x_\ell(t) \in\mathbb{R}^{2},~\ell=1,2,\dots,N$.
Since the set of equations \eqref{Eq:eq:Gener_MSF} only differs by the parameter $\lambda_\ell$, the stability problem is reduced to a single two-component delay equation 
\begin{align}
\label{Eq:eq:Gener_MSF-1}
\dot x(t)= 
        (M -\lambda P) x(t)+(M^\tau -\lambda P^\tau) x(t-\tau),
\end{align}
where $x(t) \in\mathbb{R}^{2}$ and the parameter $\lambda$ takes the values $\lambda_\ell$ of the eigenvalues of the Laplacian matrix $\mathbf{L}$. 
System \eqref{Eq:eq:Gener_MSF-1} allows a separation of the stability problem and the coupling topology given by $\mathbf{L}$. Solving the stability problem for \eqref{Eq:eq:Gener_MSF-1} with $\lambda$ as a parameter provides stability conditions for an arbitrary coupling topology. 

\subsection{Delay-independent classification of linear DDEs \label{sec:ACS}}
Using the notations $A=M-\lambda P$ and $B=M^\tau-\lambda P^\tau$, equation \eqref{Eq:eq:Gener_MSF-1} can be rewritten in a more general form as a linear DDE with one delay
\begin{align}
\label{eq:linearDDE}
        \dot x(t)=A x(t)+B x(t-\tau).
\end{align}
The corresponding characteristic equation determining the stability and eigenvalues is 
\begin{equation}
\label{eq:chareq}
        \det \left[ \mu I - A - B e^{-\mu\tau} \right] = 0.
 \end{equation}
 
To study the stability properties of \eqref{eq:linearDDE} and the roots of the characteristic equation \eqref{eq:chareq} as a function of time delay $\tau$, we will use the methods developed in \cite{Yanchuk2022b} and \cite{wang2024universal}. 
These methods introduce classes of DDEs, which are either stable for all time delays or undergo different but universal destabilization scenarios. To classify these DDEs, we will introduce the concepts of the generating polynomial and the asymptotic continuous spectrum.
\begin{definition}
\cite[Definition 2.3] {wang2024universal}  
    The  function 
        \begin{align}
            \chi_\omega(Y) := \det\left[i\omega I-A-B Y \right]
        \label{eq:Y}
        \end{align}
is called the \textbf{generating polynomial}.  $\chi_\omega(Y)$ is a polynomial with respect to both $Y\in \mathbb{C}$ and $\omega\in \mathbb{R}$.
    We also denote the roots of the generating polynomial as $Y_{j}(\omega)$, i.e.  $\chi_\omega(Y_j(\omega))=0$, where $j=1,\dots,m$.    
    $m$ is the rank of matrix $B$.
\end{definition}
Using the roots of the generating polynomial, we introduce the following concept of the \textit{asymptotic continuous spectrum}.
\begin{definition}
\cite[Definition 2.4]{wang2024universal}
\label{def:1}
    The \textbf{Asymptotic Continuous Spectrum (ACS)} is given by
	\begin{align}
			 & \Lambda_{\text{ACS}} :=  \left\{ \frac{1}{\tau}\gamma_j(\omega) + i\omega\in\mathbb{C}\ :\ \omega \in \mathbb{R}, \quad j=1,\dots,m\right\}, \label{eq:ACS-complex}\\
			& \text{where} \quad \gamma_{j}(\omega)  =-\ln\left|Y_{j}(\omega)\right|.\label{eq:ACS}
	\end{align}
\end{definition}
The ACS is shown to approximate the spectrum of DDEs of the form \eqref{eq:linearDDE} for large delays \cite{Lichtner2011,Sieber2013}. 
It consists of $m$ continuous curves in $\mathbb{C}$ in general.
In our case, we have $B=M^\tau-\lambda P^\tau$,  $\textnormal{rank}\, B =1$ and there is only one ACS curve.
\begin{figure}
\centering   
\includegraphics[width=12cm]{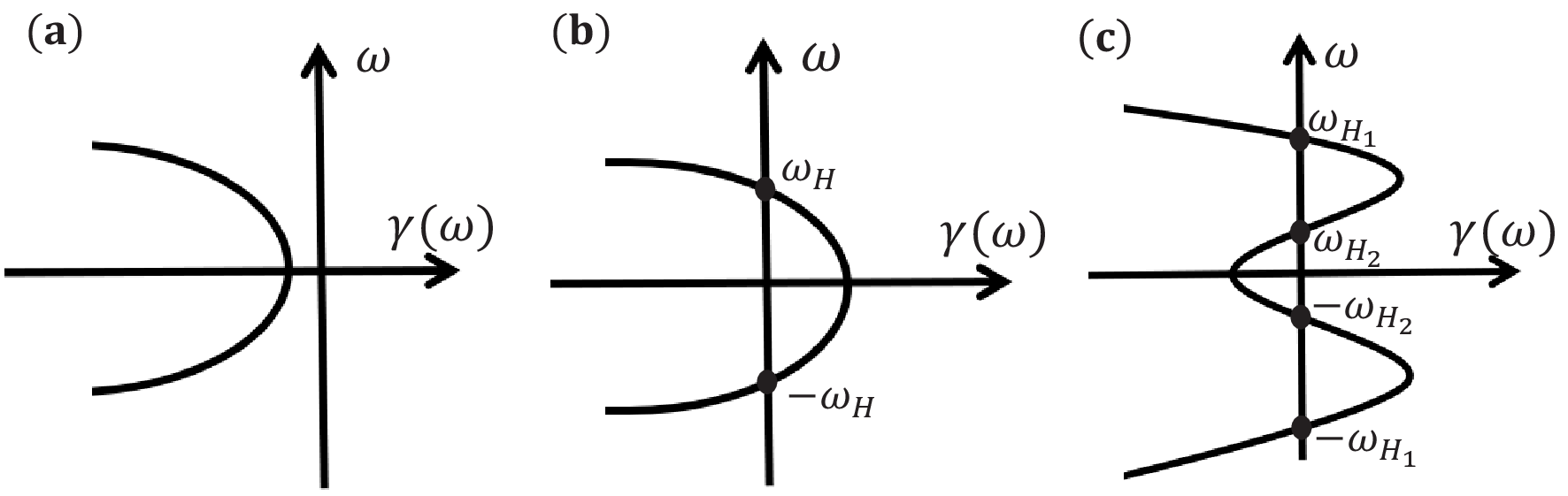}
    \caption{Schematic structure of universality classes 0--II asymptotic continuous spectrum: (a) universality class 0, (b) universality class I, (c) universality class II.
\label{Fig:ACSs}}
\end{figure}

Furthermore, we introduce the following classification of the ACS, which is closely related and will lead to the appropriate classification of the DDEs \cite{wang2024universal}:

\textit{Class 0 ACS} shows no crossing with the imaginary axis, see Fig.~\ref{Fig:ACSs}(a). 

\textit{Class I ACS} has two points $\pm i \omega_H$ where it crosses the imaginary axis, see Fig.~\ref{Fig:ACSs}(b). 

\textit{Class II ACS} exhibits four crossing points $\pm i \omega_{H_1}$ and $\pm i \omega_{H_2}$ as shown in Fig.~\ref{Fig:ACSs}(c).


Additionally, a strongly unstable spectrum $\Lambda_U$ is defined as a set of eigenvalues of matrix $A$ with positive real parts~\cite{Lichtner2011}. 
Systems with a nonempty strongly unstable spectrum are always unstable for sufficiently large delays. 
Moreover, such systems become stable for small delays if the matrix $A+B$ is stable.

Based on the above discussion, we now introduce the concepts of absolute stability and universality classes I, II, and U of DDEs \eqref{eq:linearDDE} to further describe the stability and various forms of instabilities of this system, see more details and proofs in \cite{Yanchuk2022b,wang2024universal}.
\begin{definition}
\cite[Definition 1]{Yanchuk2022b}
    System \eqref{eq:linearDDE} is \textbf{absolutely stable} if all roots $\mu$ of the characteristic equation \eqref{eq:chareq} possess negative real parts $\Re(\mu) < 0$ for all $\tau\geq 0$. 
\end{definition}
Accordingly to \cite{Yanchuk2022b}, system \eqref{eq:linearDDE} is absolutely stable iff the ACS is of class 0 and $A+B$ is stable. 
\begin{definition}
\cite[Definition 4.1]{wang2024universal}
    We define the system \eqref{eq:linearDDE}  to be of \textbf{universality class I} if its ACS is of class I.
\end{definition}
The destabilization time delays for DDEs of class I are given by \cite{wang2024universal} 
    \begin{equation}
        \label{eq:classIbif-n}
        \tau_j = \frac{1}{\omega_H}(\phi_H +2\pi j),\quad j=0,1,\dots
    \end{equation}
with $\phi_{H}=-\arg\left[Y(\omega_{H})\right] \in [0,2\pi]$
and $\omega_H>0$, which are independent on $\tau$.
Here $Y$ is defined by \eqref{eq:Y}. 
Hence, there are two stability scenarios of class I DDEs: (i) stability  for $0 \le \tau<\tau_0$ and instability for $\tau>\tau_0$ if $A+B$ is stable; and (ii) instability for all $\tau\ge 0$ if $A+B$ is unstable. Note that in case (ii) the DDE is also unstable for all positive time delays. 

\begin{definition}
\cite[Definition 5.1]{wang2024universal}
    We define the system \eqref{eq:linearDDE} to be of \textbf{universality class II} iff its ACS is of universality class II.  
\end{definition}
    The class II DDEs possess a sequence of destabilizing transitions \cite{wang2024universal} at 
    \begin{equation}
    \label{eq:classIIbifDest}
        \tau_j^{(1)} = \frac{1}{\omega_{H_1}}(\phi_{H_1} +2\pi j),
    \end{equation}
    and stabilizing transitions at 
    \begin{equation}
    \label{eq:classIIbifStab}
        \tau_{j}^{(2)} = \frac{1}{\omega_{H_2}}(\phi_{H_2} +2\pi j),
    \end{equation}    
    as time-delay varies, where $0<\omega_{H_2}<\omega_{H_1}$, and $\phi_{H_i}=-\arg\left[Y(\omega_{H_i})\right]$ are delay-independent.
     Moreover, the regions of instability and stability can alternate (known as stability islands \cite{DHuys2008,Huddy2020,Daniel2022a}), and for sufficiently large $\tau$ the system is unstable.
\begin{definition}
    System \eqref{eq:linearDDE} is said to be of  \textbf{universality class $U$},  if for all $\tau > 0$, the characteristic equation \eqref{eq:chareq} has at least one root $\mu$ with a positive real part $\Re(\mu)>0$ (unstable) and no roots occur with $\Re(\mu)=0$ (hyperbolic).
\end{definition}
According to \cite{Yanchuk2022b}, the criterion  for a linear DDE system to be of class $U$ is: ACS of type 0 and unstable $A$. In this case, the system remains unstable and hyperbolic for all positive $\tau$.

\section{The case of uncoupled agents}\label{sec:Uncoupled_Case}

\subsection{Characteristic equation, weak and strong spectrum}
We first analyze the case where all active agents receive information only from the virtual leader, which we call the case of uncoupled agents, see Fig.~\ref{Fig:example1}. 
\begin{figure}
\centering  
\includegraphics[width=1.5in,height=1.8in]{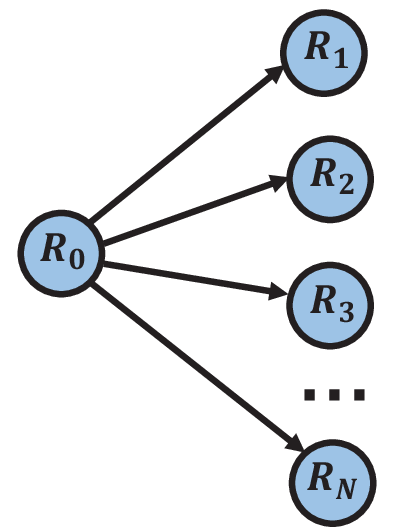}
    \caption{Schematic diagram of the uncoupled system.
    Each agent $R_1,\dots,R_N$ receives the input from the virtual leader $R_0$.
\label{Fig:example1}}
\end{figure}
When the coupling is absent, we have $\mathbf{L}=\mathbf{0}$ and the control function \eqref{eq:controller} reads
\begin{align}
\mathbf{U} = 
        - k_{0} \mathbf{e}(t) - k^\tau_{0}\mathbf{e}(t-\tau)- h_{0}\boldsymbol{\xi}(t)- h^\tau_{0}\boldsymbol{\xi}(t-\tau)+ \mathbf{U}_0(t),
\label{Eq:U_1}
\end{align}
hence system \eqref{eq:DDE-e}--\eqref{eq:DDE-xi} reduces to the following form
\begin{align}
\mathbf{\dot e} = 
        & \, \boldsymbol{\xi}, 
\label{eq:DDE_e0}
\\
\boldsymbol{\dot \xi} = 
        & - k_{0} \mathbf{e}(t) - k^\tau_{0}\mathbf{e}(t-\tau)- h_{0}\boldsymbol{\xi}(t)- h^\tau_{0}\boldsymbol{\xi}(t-\tau).
\label{eq:DDE_xi0}
\end{align}
This results in a set of uncoupled equations, each corresponding to an individual agent $i$, where $i=1,\dots,N$:
\begin{align}
\label{eq:error_example1}
\begin{array}{l}
    \dot e_i(t)=
        \xi_i(t),
    \\
    \dot \xi_i(t)=
        -k_{0}e_i(t)-k^\tau_{0}e_i(t-\tau)-h_{0}\xi_i(t)-h^\tau_{0}\xi_i(t-\tau).
\end{array}
\end{align}
Since the equations for all $i$ are identical, we can drop the index $i$.
In the vector form, the error system becomes
\begin{align}
\label{Eq:error-uncoupled}
        \dot x(t)= M x(t)+ M^\tau x(t-\tau)
\end{align}
with $M$ and $M^\tau$ are defined by \eqref{eq:MP} and $y=(e_i,\xi_i)$.
The stability of the error system \eqref{Eq:error-uncoupled} is governed by the characteristic equation
\begin{align}
\label{eq:charac_1}
        \mu^2+h_0\mu+h^\tau_0\mu e^{-\mu\tau}+k_0+k^\tau_0 e^{-\mu\tau}=0.
\end{align}

The associated ACS curve described by \eqref{eq:ACS-complex} (refer to Sec.~\ref{sec:Preliminaries}) reduces to a single expression for $\gamma(\omega)$, given by:
\begin{align}
\label{eq:ACS_1}
\gamma(\omega)=
        -\frac{1}{2}\ln\left[\frac{\left(\omega^2-k_0\right)^2+\left(\omega h_0\right)^2}{\left(k^\tau_0\right)^2+\left(\omega h^\tau_0\right)^2}\right].
\end{align}

The strongly unstable spectrum $\Lambda_U$ of system \eqref{eq:error_example1} is then given as the unstable spectrum of the instantaneous part of \eqref{eq:error_example1}, which are the roots of the characteristic equation
\begin{align}
\label{eq:F_mu}
\det \left[ \mu\, \mathrm{I} -M  \right] = 
    \det \left[\begin{array}{ccc}
		\mu &-1\\
		k_0&\mu+h_0
    \end{array}\right]=
        \mu^2+h_0\mu+k_0=0
\end{align}
with positive real parts. 

\subsection{Stability diagrams for uncoupled agents}\label{subsec3_1}
In this section, we obtain stability conditions for system \eqref{eq:DDE_e0}--\eqref{eq:DDE_xi0} (equivalently, for \eqref{Eq:error-uncoupled}), which include delay-independent classifications, as well as diagrams for specific time delays.

The following theorem gives an explicit delay-independent classification.
To shorten the notations, we denote the set of parameters of system \eqref{eq:DDE_e0}--\eqref{eq:DDE_xi0} as $p_0:=(k_0,h_0,k_0^\tau,h_0^\tau)$.

\begin{theorem}
\label{The_4-1}
    The error time-delay system  \eqref{eq:DDE_e0}--\eqref{eq:DDE_xi0} has the following delay-independent classification:
  
  \textnormal{(i)}
    The system is absolutely stable if and only if $p_0 \in S_0$, where the parameter set $S_0$ is defined as 
        \begin{equation}
        \begin{aligned}
        \label{eq:S0}
            S_0:= 
                & \{p_0:\ k_0 > |k^\tau_0| \quad \text{and} \quad  h_0 > |h^{-}_0|\}, 
            \\
                & \text{where} \quad ({h^{-}_0})^2 = 2k_0 + (h^\tau_0)^2 - 2\sqrt{(k_0)^2 - (k^\tau_0)^2}.
        \end{aligned}  
        \end{equation}

  \textnormal{(ii)}
    The system belongs to universality class I if and only if $p_0 \in S_I$ with 
\begin{equation}
\label{eq:SI}
        S_I:= \left\{ p_0:\ \left|k_0\right| < |k^\tau_0|\right\}.
\end{equation}

\textnormal{(iii)}
    The system belongs to universality class II if and only if $p_0\in S_{II}$ with
\begin{equation}
\label{eq:SII}
\begin{aligned}
        S_{II}:=  \left\{ p_0:\ \left|h_0\right| < |h^{-}_0| \right\} \cap S_{II,\mathrm{part}},
\end{aligned}
\end{equation}
where $S_{II,\mathrm{part}}$ is defined as the set of all $p_0$ for which at least one of the following conditions holds

    {\quad  a) $k_0 > |k^\tau_0|$,} 

    {\quad b)  $-\left[\left(\frac{h^\tau_0}{2}\right)^2+\left(\frac{k^\tau_0}{h^\tau_0}\right)^2\right]\leq k_0 <-|k^\tau_0|$.}

  \textnormal{(iv)}
    The system belongs to universality class $U$ if and only if $p_0\in S_{U}$, where $S_U$ is defined as a the set of all $p_0$ for which at least one of the following conditions holds

    {\quad  a) $k_0<-\left[\left(\frac{h^\tau_0}{2}\right)^2+\left(\frac{k^\tau_0}{h^\tau_0}\right)^2\right]$, }
 
    {\quad b)   $-\left[\left(\frac{h^\tau_0}{2}\right)^2+\left(\frac{k^\tau_0}{h^\tau_0}\right)^2\right]\leq k_0 <-|k^\tau_0|$, and $ \left|h_0\right| > |h^{-}_0|$,}
  
    {\quad c) $k_0 >|k^\tau_0|$ and $h_0<- |h^{-}_0|$.}
\end{theorem}

\begin{proof}
Firstly, we note that the strongly unstable spectrum is not empty (the characteristic equation \eqref{eq:F_mu} for the instantaneous part of the system has at least one root with positive real part) if and only if $k_0<0$ or $h_0<0$. 

The absolute stability criterion is that the strongly unstable spectrum is absent (i.e., $k_0\ge0$ and $h_0\ge0$),
and the ACS belongs to class 0, i.e, it does not cross the imaginary axis~\cite{Yanchuk2022b}. The following equation determines the points, at which the curve of the ACS \eqref{eq:ACS_1} crosses the imaginary axis and $\gamma(\omega)=0$:
\begin{align}
        \omega^4+\left[ \left(h_0\right)^2-2k_0-\left(h^\tau_0\right)^2\right]\omega^2+ \left(k_0\right)^2-\left(k^\tau_0\right)^2=0.
\label{eq:criticalCond}
\end{align}
With $\nu=\omega^2$, 
equation \eqref{eq:criticalCond} becomes  quadratic in $\nu$, and 
\begin{align}
   \nu_{\pm}=-\frac{\left(h_0\right)^2-2k_0-\left(h^\tau_0\right)^2}{2}
   \pm\sqrt{\left[ \frac{\left(h_0\right)^2-2k_0-\left(h^\tau_0\right)^2}{2}\right]^2-\left[ \left(k_0\right)^2-\left(k^\tau_0\right)^2\right]}.
   \label{eq:nu}
\end{align}
The ACS is of class 0 if this quadratic equation has no real non-negative roots, which corresponds to the condition that the spectrum does not intersect the imaginary axis. This can occur in two cases:
\begin{subequations} \label{eq:UC0_conditions}
\begin{align}
\label{eq:UC0_condition1}
\nu_{\pm}~\text{are not real:}
    \left\{
    \begin{array}{cc}
    \left(k_0\right)^2-\left(k^\tau_0\right)^2 >0,\\
    \Delta=\left[ \left(h_0\right)^2-2k_0-\left(h^\tau_0\right)^2\right]^2-4\left[ \left(k_0\right)^2-\left(k^\tau_0\right)^2\right]<0
    \end{array}
    \right.
\end{align}
or
\begin{align}
\label{eq:UC0_condition2}
\nu_{\pm}~\text{are negative real solutions:}
    \left\{
    \begin{array}{ccc}
    \left(k_0\right)^2-\left(k^\tau_0\right)^2 >0,\\
    \Delta=\left[ \left(h_0\right)^2-2k_0-\left(h^\tau_0\right)^2\right]^2-4\left[ \left(k_0\right)^2-\left(k^\tau_0\right)^2\right]\geq0,\\
    \left(h_0\right)^2-2k_0-\left(h^\tau_0\right)^2>0.
    \end{array}
    \right.
\end{align}
\end{subequations}
By combining the conditions in \eqref{eq:UC0_condition1} and \eqref{eq:UC0_condition2}, the ACS is of class 0 if and only if
\begin{align*}
        \left(k_0\right)^2-\left(k^\tau_0\right)^2 >0 
\quad\text{and}\quad 
\left(h_0\right)^2>2k_0+\left(h^\tau_0\right)^2 - 2 \sqrt{\left(k_0\right)^2-\left(k^\tau_0\right)^2},    
\end{align*}
which can be simplified as
\begin{align*}
        k_0> \left|k^\tau_0\right| 
\quad\text{and}\quad 
        h_0 > |h_0^-|,  
\end{align*}
where 
\begin{align*}
(h_0^-)^2 = 
        2k_0+\left(h^\tau_0\right)^2 - 2 \sqrt{\left(k_0\right)^2-\left(k^\tau_0\right)^2}. 
\end{align*}
This proves \eqref{eq:S0} and statement (i) of the theorem. 

Next, the condition for the system \eqref{eq:DDE_e0}--\eqref{eq:DDE_xi0}  to be of universality class I is that the polynomial \eqref{eq:criticalCond} has exactly two real roots $\pm\omega_H$ and $\gamma(0)>0$. Then, from equation \eqref{eq:nu}, it follows that the criterion for satisfying the inequality $0 < \nu_{+} = \omega_H^2$ is 
\begin{align*}
        \left(k_0\right)^2-\left(k^\tau_0\right)^2 <0,
\end{align*}
or, equivalently, $\left|k_0\right|<\left|k^\tau_0\right|$.
This proves statement (ii) of the theorem. 

The system belongs to the universality class II if the polynomial \eqref{eq:criticalCond} has two pairs of real roots  $\pm \omega_{H_1}$ and $\pm \omega_{H_2}$. These exist if the discriminant is positive and if both roots $\nu_\pm=\omega_{H_{1,2}}^2$, given by equation \eqref{eq:nu}, are positive. This holds under the following conditions:
\begin{align*}
  \left\{  \begin{array}{ccc}
        |k_0|>|k^\tau_0|,\\
         \left(h_0\right)^2-2k_0-\left(h^\tau_0\right)^2<0,\\
         \Delta=\left[ \left(h_0\right)^2-2k_0-\left(h^\tau_0\right)^2\right]^2-4\left[ \left(k_0\right)^2-\left(k^\tau_0\right)^2\right]>0.
    \end{array}\right.
\end{align*}
Straightforward calculations lead to one of the following conditions holding:
\begin{subequations}
    \label{eq:UCII_conditions}
    \begin{align}
        \label{eq:UCII_condition1}
        \left\{  \begin{array}{ccc}
        k_0>|k^\tau_0|,\\
          \left(h_0\right)^2<2k_0+\left(h^\tau_0\right)^2 - 2 \sqrt{\left(k_0\right)^2-\left(k^\tau_0\right)^2}.
    \end{array}\right.
    \end{align}
    \begin{align}
        \label{eq:UCII_condition2}
        \left\{  \begin{array}{ccc}
        k_0<-|k^\tau_0|,\\
          \left(h_0\right)^2<2k_0+\left(h^\tau_0\right)^2 - 2 \sqrt{\left(k_0\right)^2-\left(k^\tau_0\right)^2},\\
          0\leq 2k_0+\left(h^\tau_0\right)^2 - 2 \sqrt{\left(k_0\right)^2-\left(k^\tau_0\right)^2}.
    \end{array}\right.
    \end{align}
\end{subequations}
Combining  the conditions in \eqref{eq:UCII_condition1} and \eqref{eq:UCII_condition2} yields
\begin{align*}
        \left|h_0 \right|< \left|h_0^-\right|,
\end{align*}
and either of the following conditions is satisfied:
\begin{align*}
        k_0>\left|k^{\tau}_0\right|,~
\text{or},~ 
        -\left[\left(\frac{h^\tau_0}{2}\right)^2+\left(\frac{k^\tau_0}{h^\tau_0}\right)^2\right]< k_0 <-|k^\tau_0|.
\end{align*}
Note that in the case $h_0^\tau =0$, the left-hand side can be considered to be $-\infty$ and $k_0$ is not bounded from below.
This proves statement (iii) of the theorem. 

Finally, for the system to be of universality class U, we require that the strongly unstable spectrum is not empty (unstable), i.e.,
$k_0<0$ or $h_0<0$, and the ACS does not cross the imaginary axis (hyperbolic).
From the no-crossing conditions in Eqs.~\eqref{eq:UC0_condition1} and \eqref{eq:UC0_condition2}, we first consider
\begin{align*}
 2k_0+\left(h^\tau_0\right)^2 - 2 \sqrt{\left(k_0\right)^2-\left(k^\tau_0\right)^2}<0,  
\end{align*}
which leads to 
 \begin{align*}
      k_0<-\left[\left(\frac{h^\tau_0}{2}\right)^2+\left(\frac{k^\tau_0}{h^\tau_0}\right)^2\right].
 \end{align*}
Further,  if
\begin{align*}
    (h_0^-)^2 = 2k_0+\left(h^\tau_0\right)^2 - 2 \sqrt{\left(k_0\right)^2-\left(k^\tau_0\right)^2}\geq0,
\end{align*}
we get
\begin{align*}
        -\left[\left(\frac{h^\tau_0}{2}\right)^2+\left(\frac{k^\tau_0}{h^\tau_0}\right)^2\right]\leq k_0 <-|k^\tau_0|
    \quad\text{and}\quad 
        |h_0| > |h^{-}_0|,    
\end{align*}
or
\begin{align*}
        k_0 >|k^\tau_0|
\quad\text{and}\quad
        h_0<- |h^{-}_0|.    
\end{align*}
This confirms statement (iv) of the theorem. 
\end{proof}

Figure~\ref{Fig:Abs_stable_Region} illustrates the classification provided by Theorem~\ref{The_4-1} in the $(k_0,h_0)$-plane for fixed parameter values $k^{\tau}_0=1.5$ and $h^{\tau}_0=-3$. 
\begin{figure}
\centering	      
\includegraphics[width=9cm]{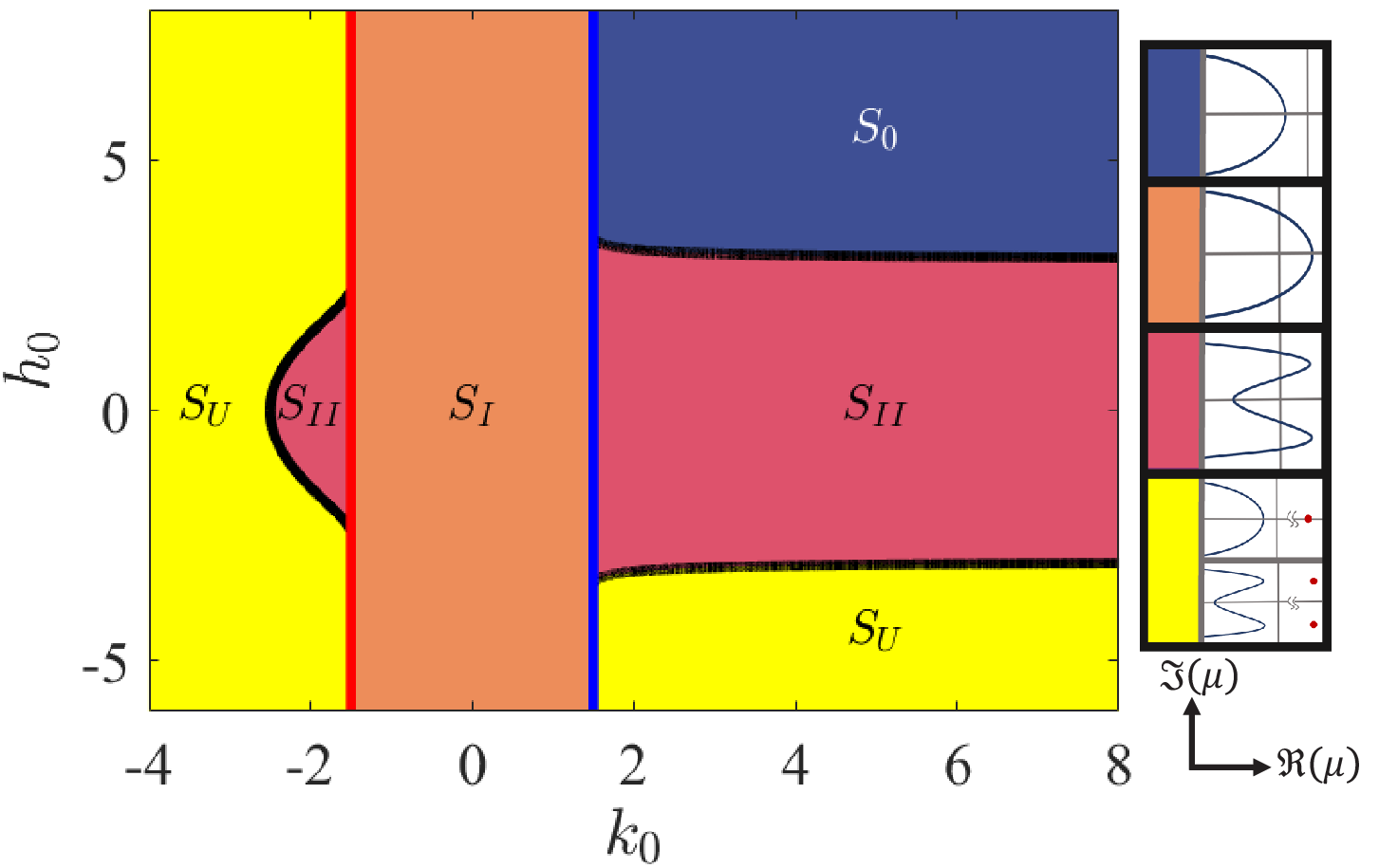}
      \caption{Classification of error system~\eqref{eq:error_example1} in the $(k_0,h_0)$-plane for fixed $k^\tau_0=1.5, h^\tau_0=-3$, according to Theorem~\ref{The_4-1}. 
      The colors show class 0 (blue),
      class I (orange), class II (red), and class U (yellow), respectively. 
      The boundaries are given by $(h_0^-)^2 = 2k_0+\left(h^\tau_0\right)^2 - 2 \sqrt{\left(k_0\right)^2-\left(k^\tau_0\right)^2}$ (black), $k_0=k_0^\tau$ (red), and $k_0=0$ or $k_0=-k_0^\tau$ (blue) lines.
      The right panel sketches the asymptotic continuous spectrum (blue curves) and strongly unstable spectrum (red symbols) for the different classes.
\label{Fig:Abs_stable_Region}}
\end{figure}
The right panel of Fig.~\ref{Fig:Abs_stable_Region} also sketches the asymptotic continuous spectrum (blue curves) and the strongly unstable spectrum (red symbols) for the different classes.
This classification is independent of delay, with each region corresponding to either a stable state or a specific bifurcation scenario as the time delay varies \cite{wang2024universal}, as we will illustrate below. 

Figure~\ref{Fig:ACS_Curves} illustrates the spectrum of Eq.~\eqref{eq:error_example1} for fixed time delays $\tau=20$ and different parameter regions corresponding to the classification from Fig.~\ref{Fig:Abs_stable_Region}. 
\begin{figure}
\centering	
\includegraphics[width=11cm]{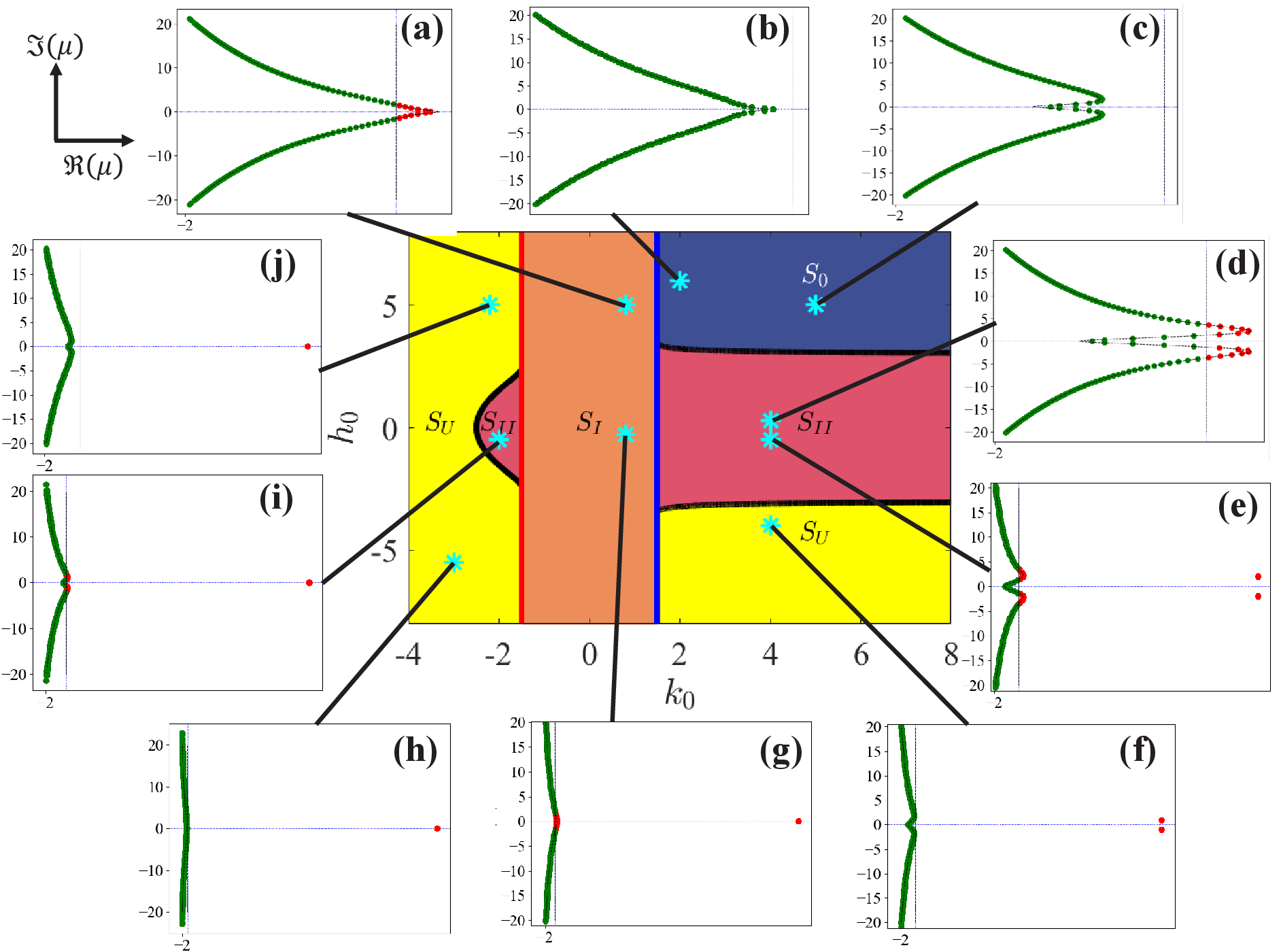}
    \caption{The spectrum of the  DDE \eqref{eq:error_example1} corresponding to the different regions in Fig.~\ref{Fig:Abs_stable_Region}, where $k_0^\tau=1.5$, $h_0^\tau=-3$. 
    Panels (a) to (j) show the spectrum  for fixed time-delay $\tau=20$ and representative parameters from different regions in the $(k_0,h_0)$-plane. 
    The green and red dots show the stable and unstable points, respectively, from the spectrum, and the solid lines denote the curves of ACS.
\label{Fig:ACS_Curves}}
\end{figure}
So one can observe the stable spectrum in the region $S_0$, the unstable and hyperbolic spectrum in region $S_U$, and the multiple unstable characteristic roots in the regions $S_I$ and $S_{II}$ with relatively small real parts. 

The delay-independent stability region is given by $S_0$. 
However, the system may also be stable in the other regions (except $S_U$) for small or intermediate delays. 
The stabilization mechanisms in the other regions are determined by specific sequences of bifurcations as described in \cite{wang2024universal}. 
To find the exact values of these bifurcations, we substitute $\mu=i\omega$ into the characteristic equation \eqref{eq:charac_1},
\begin{align}
        -\omega^2+ih_0\omega
        +ih^\tau_0\omega\left[\cos(\omega\tau)-i\sin(\omega\tau)\right]
        +k_0+k^\tau_0\left[\cos(\omega\tau)-i\sin(\omega\tau)\right]=0.
\label{eq:Pur_Chara_1}
\end{align}
By separating the real and imaginary parts of \eqref{eq:Pur_Chara_1}, we have
\begin{align*}
        -\omega^2+h^\tau_0\omega\sin(\omega\tau)+k_0+k^\tau_0\cos(\omega\tau)=0, 
   \\
        h_0\omega+h^\tau_0\omega\cos(\omega\tau)-k^\tau_0\sin(\omega\tau)=0,
\end{align*}
which can be solved for $k_0(\omega)$ and $h_0(\omega)$ to obtain the stability boundary parametrically in the $(k_0,h_0)$ parameter plane:
\begin{align}
    k_0(\omega)&=
            \omega^2-h^\tau_0\omega\sin(\omega\tau)-k^\tau_0\cos(\omega\tau),
    \label{eq:k0_h0-1} \\
    h_0(\omega)& =
            \frac{1}{\omega}k^\tau_0\sin(\omega\tau)-h^\tau_0\cos(\omega\tau). 
    \label{eq:k0_h0-2}
\end{align}
The stability boundaries for fixed time delays $\tau=2$ and $\tau=20$ are shown in Fig.~\ref{Fig:Stable_Region}.
\begin{figure}
\centering	        
\includegraphics[width=14.5cm]{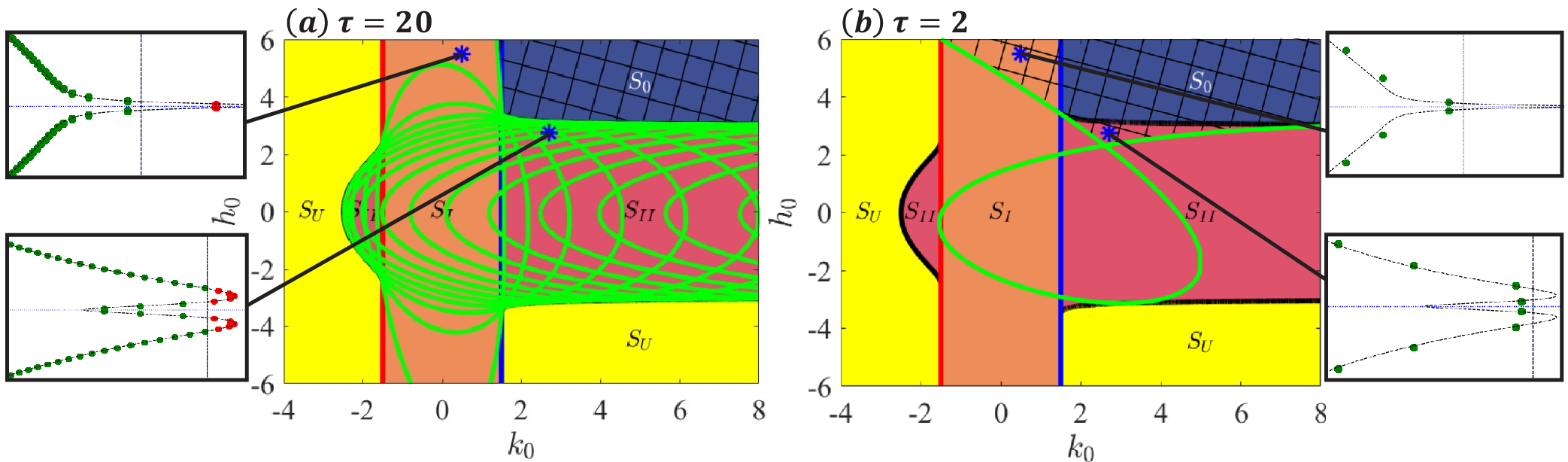}
    \caption{Bifurcation diagrams in the parameter space $(k_0,h_0)$ for time delays $ (a)~\tau=20$ and $(b)~\tau=2$  ($k^\tau_0=1.5$ and $h^\tau_0=-3$ fixed for the both diagrams). 
    The colors have the same meaning as in Fig.~\ref{Fig:ACS_Curves}. 
    The green lines are bifurcation lines. 
    The stability regions are marked by the checkered area with black skewed grid lines. 
    Sub-figures illustrate the spectra and stabilization mechanisms for the different points as the delay $\tau$ decreases.
\label{Fig:Stable_Region}}
\end{figure}
The stability regions are indicated by the cross-hatched pattern.
We observe that the stability region includes the whole region $S_0$ and also some parts of the other regions, depending on the time delay. 
For larger delays, the stability region shrinks to $S_0$, while for smaller delays it becomes larger. 
We also see how the bifurcation lines intersect in the region $S_{II}$, indicating the presence of two pairs of purely imaginary characteristic roots \cite{wang2024universal}. 
Note that degenerate points can only appear in the parameter region $S_{II}$.

\subsection{Examples with vanishing coupling coefficients \label{sec:examples}}
Here we analyze four specific scenarios that illustrate the impact of incomplete communication on the dynamics of active agents.
Each example highlights some communication constraints and their influence on the stability of motion and formation of agents.

\textit{Example 1}: 
    The velocity information is not transmitted from the virtual leader to the active agents, i.e.,  $h_0= h^\tau_0= 0$.
    The other parameters are fixed as $k_0=6$, $k^\tau_0=0.3$. 
    According to Theorem~\ref{The_4-1}, under these conditions, the system~\eqref{eq:error_example1} falls into universality class II.
    Figure~\ref{Fig:CharacCurve}(a) illustrates the stability of this system, represented by the maximal real parts of the eigenvalues, as the delay $\tau$ increases. 
    We observe the intervals in $\tau$, where the system gains stability, while it is unstable for all $\tau \gtrsim 25$.
    This is an example of stability islands appearing in the parameter space, similar to those in \cite{DHuys2008,Huddy2020,Daniel2022a}.
\begin{figure}
\centering		
    \subfigure{\includegraphics[width=5.5cm]{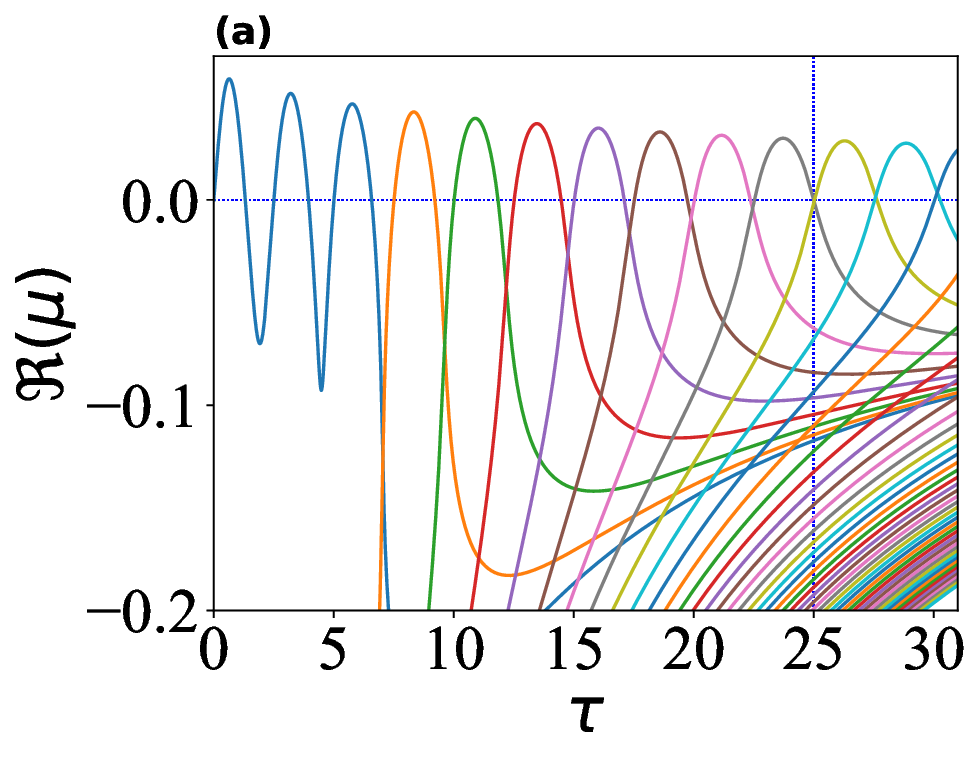}}
    \subfigure{\includegraphics[width=5.6cm]{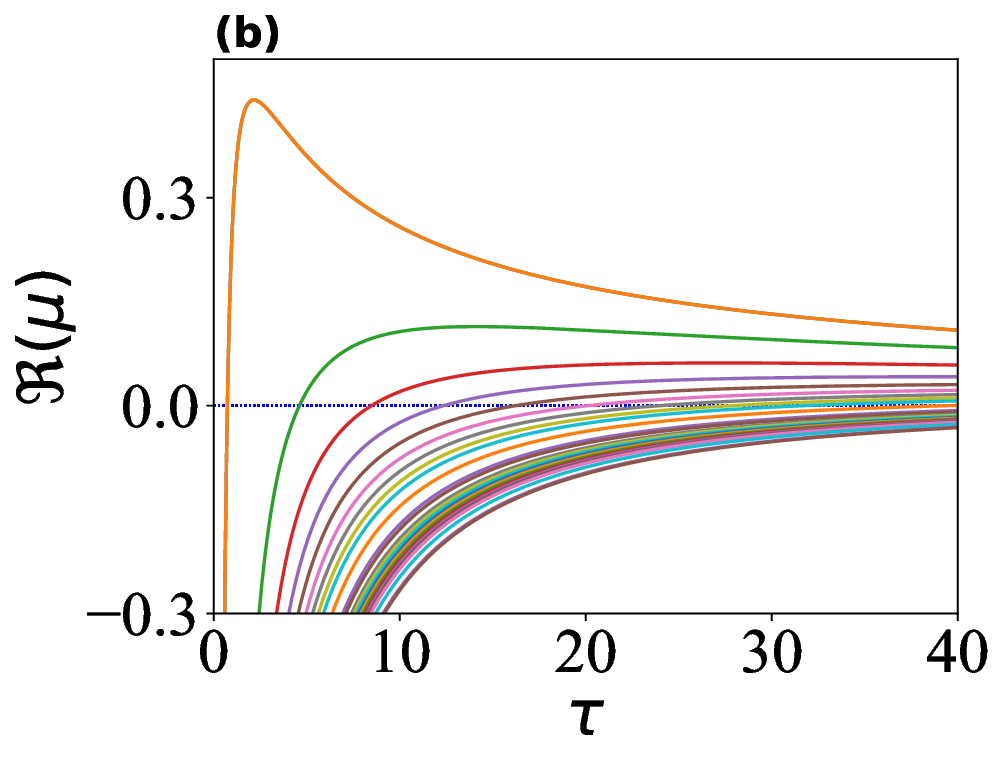}}
    \caption{(a) Maximal real parts of the eigenvalues $\Re(\mu)$ as a function of delay $\tau$ for parameter values  $h_0=h^\tau_0=0$, $k_0=6$, $k^\tau_0=0.3$ as in Example 1.
    (b) Maximal real parts of the eigenvalues $\Re(\mu)$ as a function of delay $\tau$ for  parameter values $k_0=h_0=0$, $k^\tau_0=1$, $h^\tau_0=1.5$, as in Example 2.
\label{Fig:CharacCurve}}
\end{figure}



\textit{Example 2}: 
    Consider the case $k_0=h_0=0$, i.e., the active agents communicate with time delays only. 
    Theorem~\ref{The_4-1} implies that then the system~\eqref{eq:error_example1} is classified as belonging to class I.
    Figure~\ref{Fig:CharacCurve}(b) shows a stability region for small delays and
    an increasing number of unstable eigenvalues with increasing $\tau$.

\subsection{Numerical simulation for an example with three agents}
We now present numerical simulations of a system with three active agents, which communicate only with the virtual leader, see Fig.~\ref{Fig:Motion}.
\begin{figure}
\centering
\includegraphics[width=14cm]{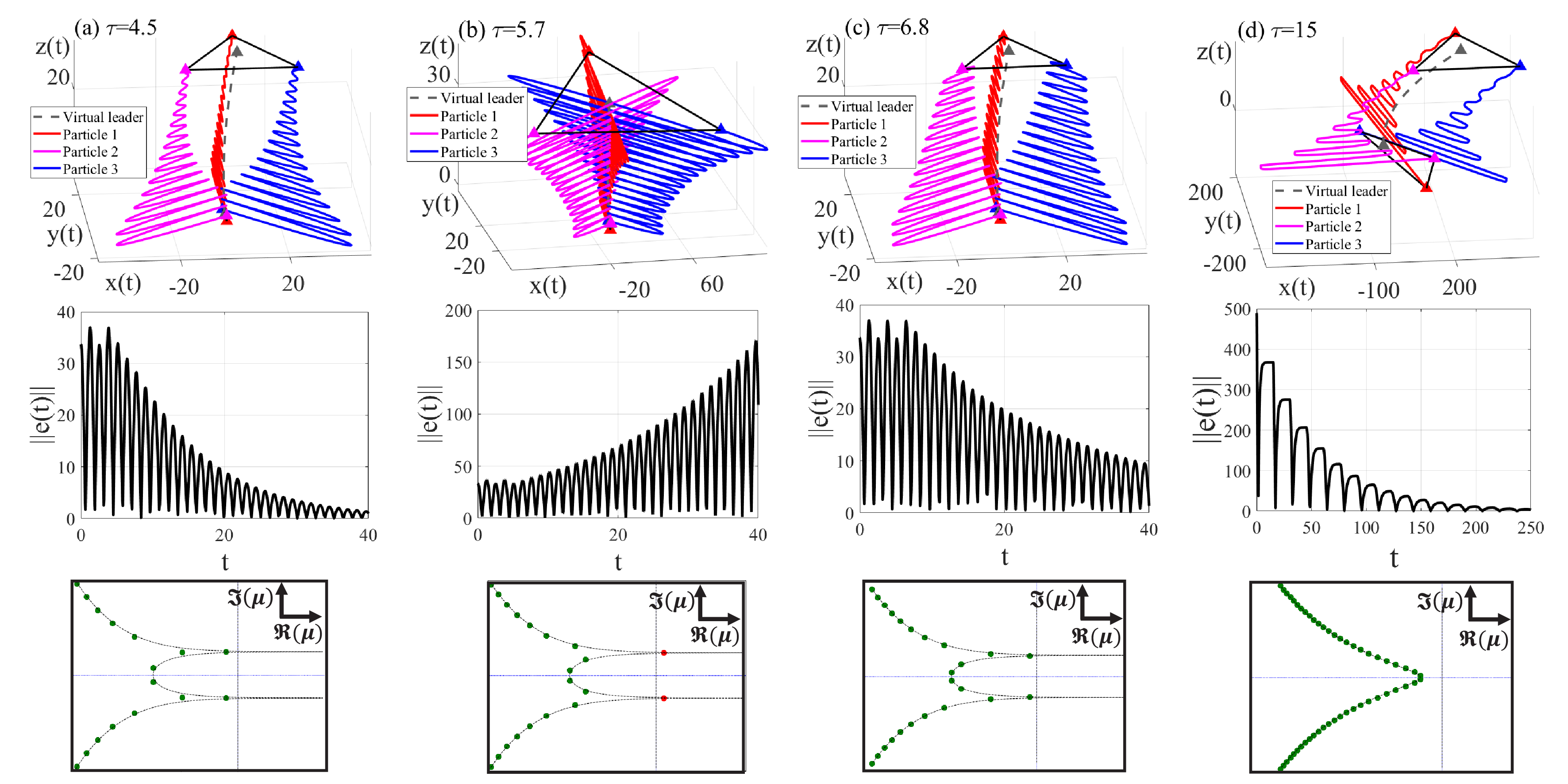}
    \caption{Illustration of stable and unstable formations of active agents when there is no coupling between them.  
    Including: the formation motion of active agents, as given by Eqs.~\eqref{eq:Motion_AP_a}--\eqref{eq:Motion_AP_b}; the time dependence of the error, as given by Eqs.~\eqref{eq:DDE_e0}--\eqref{eq:DDE_xi0}; the asymptotic continuous spectrum of the error system, as given by Eq.~\eqref{eq:ACS-complex}.
    (a)--(c) parameters are as in example 1 from Fig.~\ref{Fig:example1} ( $k_0=6$, $k^\tau_0=0.3$, $h_0=h^\tau_0=0$, $c=1$).
    (a) All active agents stabilize to the desired pattern formation at small values of $\tau=4.5$; 
    (b) all agents undergo repelling oscillatory motion as delay increases to $\tau=5.7$; 
    (c) all agents stabilize to the desired formation configuration as further increases to  $\tau=6.8$.
    (d) Parameters correspond to the region $S_0$ with stable spectrum: ($k_0 = 2$, $k_0^{\tau} = 1.5$, $h_0 = 3$, $h_0^{\tau} = 1.2$, $c=10$), i.e., the agents achieve the desired formation configuration for all values of $\tau$. 
\label{Fig:Motion}}
\end{figure}
We consider four cases. 
For cases 1, 2, and 3 we choose the parameters as in example 1 from Fig.~\ref{Fig:CharacCurve}(a), and for case 4 we choose $k_0 = 2$, $h_0 = 3$,  $ k_0^{\tau} = 1.5 $, and $h_0^{\tau} = 1.2$. 
For all cases, the target trajectory can be set as follows:
\begin{align}\nonumber
\mathbf{R_T}=&
        1_3\otimes R_0(t)+\mathbf{s}
    \\ \nonumber=&
        1_3\otimes\left[0.005(t^2+1),\ 0.5t,\ 0.8t\right]\\&+c\left([0,-10, 0], \ [20, 10, 0],\ [-20, 10,0]\right),~c\in\mathbb{Z^{+}}
\label{Eq:target_trajectory}
\end{align} 
which generates a parabola trajectory for all agents, while a desired isosceles triangular formation configuration. 
To clarify the motion trajectory, we consider that cases 1, 2, and 3 have $c=1$, and case 4 has $c=10$.

Figure~\ref{Fig:Motion} presents the simulation results, which illustrate the formation dynamics of the active agents governed by Eqs.~\eqref{eq:Motion_AP_a}--\eqref{eq:Motion_AP_b}. 
The control input applied to each agent is defined as
\begin{align*}
U_{i} =&
        -k_0\left[R_i(t)-\left(R_0(t)+s_i\right)\right]-k^\tau_0\left[R_i(t-\tau)-\left(R_0(t-\tau)+s_i\right)\right]
    \\&
        -h_0\left[V_i(t)-V_0(t)\right]-h^\tau_0\left[V_i(t-\tau)-V_0(t-\tau)\right]+ U_0(t).
\end{align*}  
Additionally, the figure depicts the time evolution of the formation error, as characterized by Eqs.~\eqref{eq:DDE_e0}--\eqref{eq:DDE_xi0}, 
and the ACS  for the error system~\eqref{eq:ACS-complex}.

For cases 1 to 3, we find that with increasing $\tau$ the system makes an interesting transition from a stable formation configuration to an unstable one and then back again to a stable formation.  
More explicitly, for $\tau=4.5$ (case 1), all agents asymptotically reach the desired formation configuration, see Fig.~\ref{Fig:Motion}(a). 
As $\tau$ increases and crosses a critical instability threshold and at $\tau=5.7$ (case 2), the desired formation configuration is not reached and the agents diverge. 
The agent trajectories transition to oscillations with exponentially increasing amplitudes and become increasingly unstable, see Fig.~\ref{Fig:Motion}(b).
Then, as $\tau$ continues to increase, the agents regain the ability to achieve the desired formation configuration after crossing a critical stability threshold as shown for $\tau=6.8$ (case 3) in Fig.~\ref{Fig:Motion}(c). 
This cycle repeats as $\tau$ continues to change. 
For case 4, the spectrum is in the absolutely stable region (i.e., class 0), it maintains a stable motion, and the desired formation can always be obtained, for all values of $\tau$, as shown in Fig.~\ref{Fig:Motion}(d). The converge of the trajectory has a square-wave pattern, often observed in DDEs with the type I continuous spectrum, see e.g. \cite{stohrSquareWavesBykov2023}.

\section{Stability of coupled agents}\label{sec:Coupled_Case}

We now study the case in which agents exchange information with each other, as well as with the virtual leader, see Fig.~\ref{Fig:coupled_example}.
\begin{figure}
\centering  
\includegraphics[width=2.8in,height=1.8in]{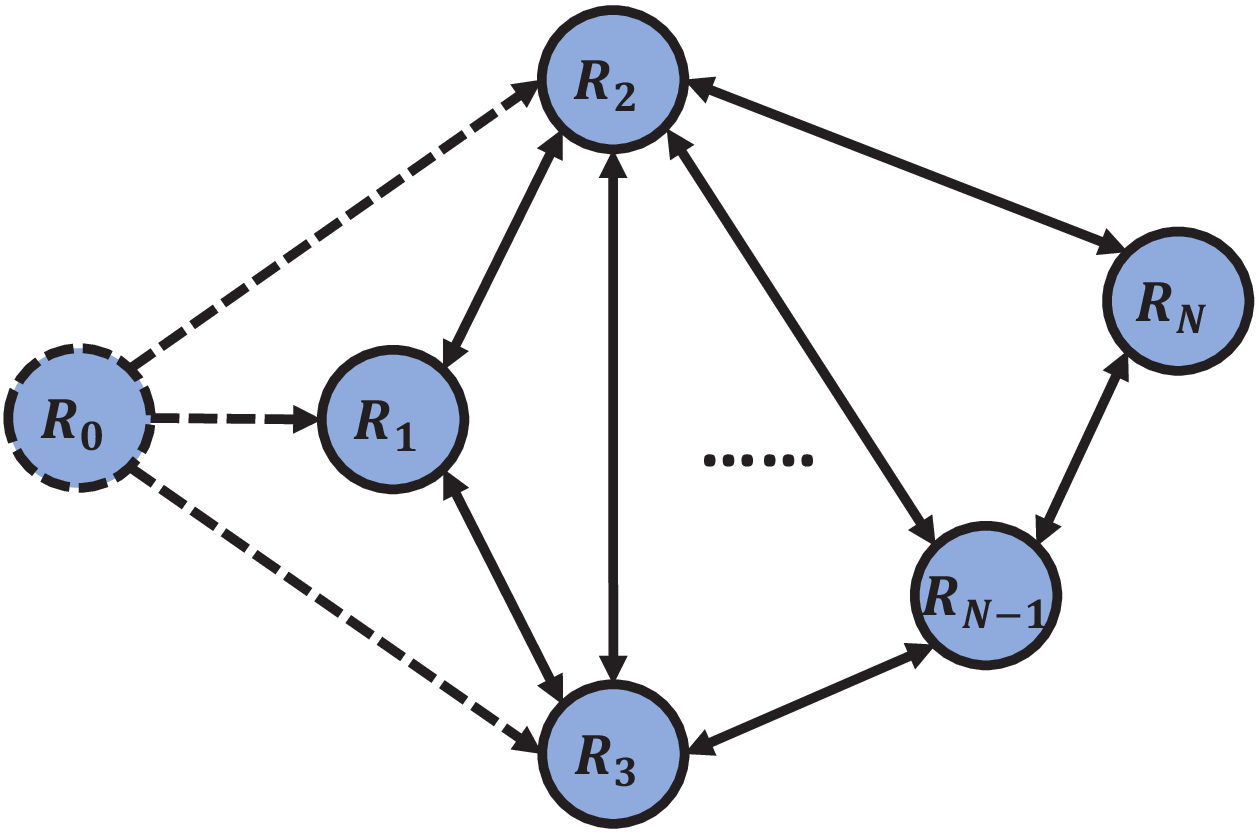}
    \caption{Schematic diagram of the coupled system. Active agents $R_1,\dots, R_N$ exchange information with each other, and receive information from the virtual leader.
\label{Fig:coupled_example}}
\end{figure}
In this case, the Laplacian  $\mathbf{L}$ in the control function \eqref{eq:controller} is not vanishing, leading to the the error system \eqref{eq:DDE-e}--\eqref{eq:DDE-xi}. 
Hence, the MSF approach from Sec.~\ref{sec:MSF} can be applied to assess the stability, and the stability problem is reduced via equation \eqref{Eq:eq:Gener_MSF-1} to
\begin{align}
\label{Eq:eq:Gener_MSF-1-again}
\dot x(t)= 
        (M -\lambda_\ell P) x(t)+(M^\tau -\lambda_\ell P^\tau) x(t-\tau),
\end{align}
where $\lambda_\ell$, $\ell=1,\dots,N$ are the eigenvalues of $\mathbf{L}$. 
In this way, the stability problem is reduced to the stability of each individual mode corresponding to different $\lambda_\ell$. 

The reduced system \eqref{Eq:eq:Gener_MSF-1-again} allows to describe the structure of the spectrum of the full error system based on the spectra of individual modes. 

\subsection{Stability of formation with symmetric coupling} 

We first consider the case of the reciprocal coupling (symmetrical topology) leading to real eigenvalues $\lambda_\ell$ of Laplacian $\mathbf{L}$. 
The main observation here is that equation \eqref{Eq:eq:Gener_MSF-1-again} has the same form as equation \eqref{Eq:error-uncoupled} for the system coupled only to the virtual leader, but the matrices $M$ and $M^\tau$ are replaced by $M-\lambda_\ell P$ and $M^\tau-\lambda_\ell P^\tau$, respectively. 
As a result, Theorem~\ref{The_4-1} can be applied to each coupling mode separately, and each mode can be classified as being in one of the universality classes 0, I, II, or U. 
Therefore, the spectrum of the whole system  \eqref{eq:DDE-e}--\eqref{eq:DDE-xi} is a union of the spectra of individual modes, and it can be described by the following notation $\text{0}_i \text{I}_j \text{II}_m \text{U}_s$, where $i,j,m,$  and $s$ denote the number of modes having the spectrum of class 0, I, II, or U, respectively, see Fig.~\ref{Fig:union_spectra} for illustration. 
\begin{figure}
\centering  
\includegraphics[width=13cm]{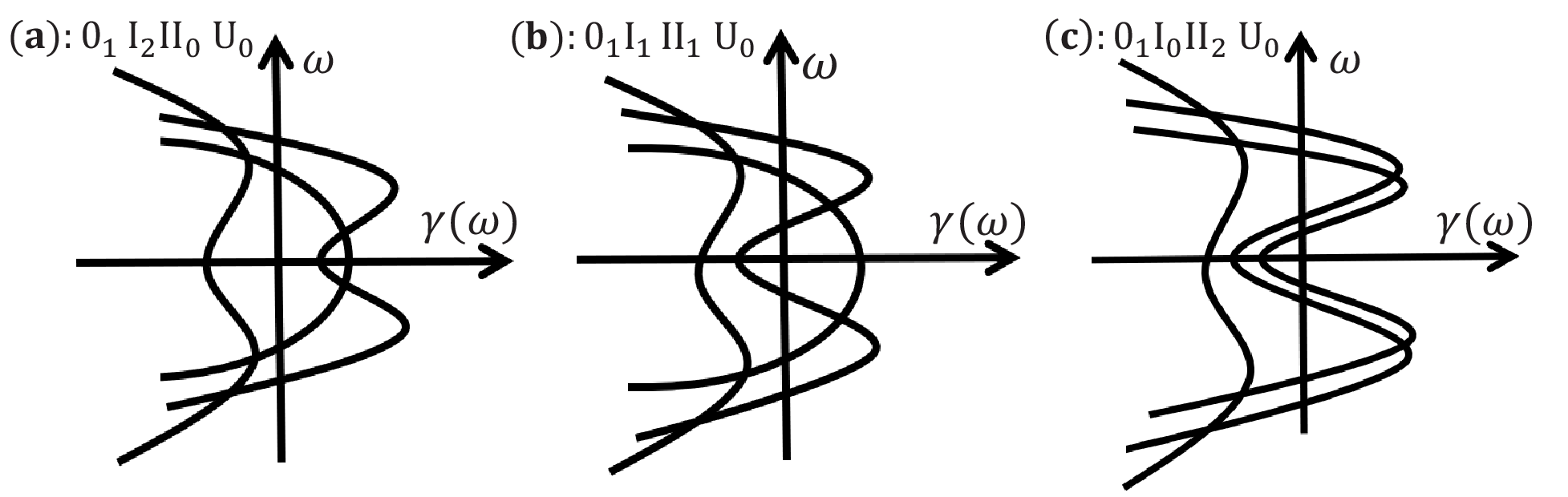}
    \caption{Schematic diagram of the asymptotic continuous spectrum of the universal class $\text{0}_i\text{I}_j\text{II}_m\text{U}_s$: 
    (a) class $\text{0}_1\text{I}_2\text{II}_0\text{U}_0$, 
    (b) class $\text{0}_1\text{I}_1\text{II}_1\text{U}_0$, 
    (c) class $\text{0}_1\text{I}_0\text{II}_2\text{U}_0$.
\label{Fig:union_spectra} }
\end{figure}

Let us analyze the special structure of the parameter dependency in system \eqref{Eq:eq:Gener_MSF-1-again}. 
For this, we use our notation $p_0=(k_0,h_0,k_0^\tau,h_0^\tau)$ and additionally define $\overline p:=(k,h,k^\tau,h^\tau)$. 
Since the matrices in \eqref{Eq:eq:Gener_MSF-1-again} are 
\begin{align*}
&M -\lambda_\ell P =-\left[\begin{array}{cccc}
		0& -1\\
		k_0+\lambda_\ell k& h_0+\lambda_\ell h
	\end{array}\right],\\
&M^\tau -\lambda_\ell P^\tau = -\left[\begin{array}{cccc}
		0& 0\\
		k^\tau_0+\lambda_\ell k^\tau& h^\tau_0+\lambda_\ell h^\tau
	\end{array}\right],
\end{align*}
we observe that system \eqref{Eq:eq:Gener_MSF-1-again} only depends on $p_0+\lambda_\ell \overline p$.
This means that the changes in the parameters $\lambda_\ell$ correspond to a line in the parameter space through the point $p_0$ and a direction given by $\overline p$. 
More specifically, the following corollary holds.
\begin{corollary}
\label{Cor:5-1}
     The following statements holds true for the coupled system \eqref{eq:DDE-e}--\eqref{eq:DDE-xi}. 

  \textnormal{(i)}
      The spectrum of eigenvalues of \eqref{eq:DDE-e}--\eqref{eq:DDE-xi} has the structure $\text{0}_i \text{I}_j \text{II}_m \text{U}_s$ with $i,j,m,s\ge 0$ and $i+j+m+s = N$. 

  \textnormal{(ii)}
      The system is absolutely stable if and only if $p_0 + \lambda_\ell \overline{p} \in S_0$ for all $\ell=1,\dots,N$, where $S_0$ is defined in \eqref{eq:S0}.

  \textnormal{(iii)}
      The system is unstable and hyperbolic if there exist $\ell$ such that  $p_0 + \lambda_\ell \overline{p} \in S_U$,  where $\ell=1,\dots,N$, and the set $S_U$ is defined  in Theorem \ref{The_4-1}(iv). 

\end{corollary} 

Figure~\ref{Fig:space_P} illustrates how the stability and the spectrum structure in the coupled system can be obtained using Corollary~\ref{Cor:5-1} and the results for the uncoupled system. 
\begin{figure}
\centering  
\includegraphics[width=5.2in,height=2.1in]
{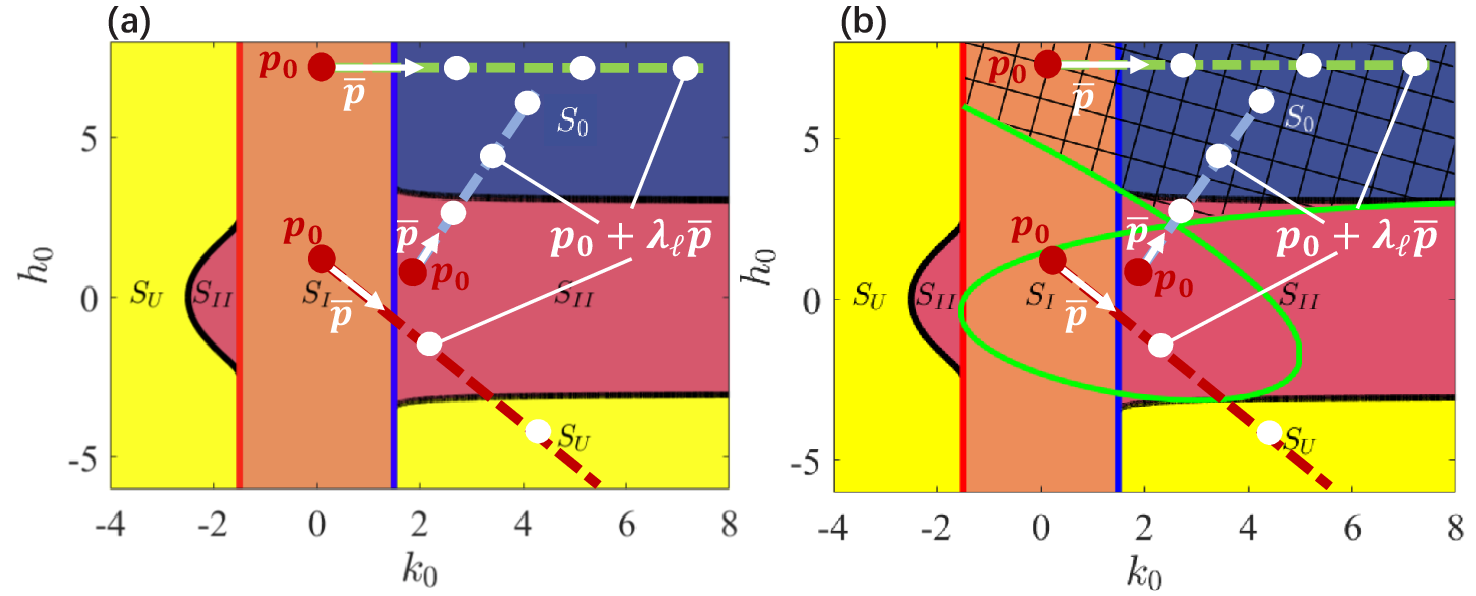}
    \caption{Illustration of the stability of the transverse modes for different parameter values and eigenvalues of the Laplacian matrix. 
    (a) Delay-independent case (large $\tau$); 
    (b) small delay case ($\tau = 2$). 
    Each dashed line is defined by $p_0 + \lambda_\ell \overline{p}$, where $p_0=(k_0,h_0,1.5,-3)$ and $\overline{p}=(k,h,0,0)$: 
    green  $p_0 = (0, 7, 1.5, -3)$, $\overline{p} = (2, 0, 0, 0)$; 
    blue $p_0 = (1.8, 0.5, 1.5, -3)$, $\overline{p} = (1, 1, 0, 0)$; 
    and red $p_0 = (0, 1, 1.5, -3)$, $\overline{p} = (2, -2, 0, 0)$.
\label{Fig:space_P}}
\end{figure}
The figure shows two cases: large and intermediate delays. 
Each dashed line is determined by the point $p_0=(k_0,h_0,1.5,-3)$ and the direction $\overline{p}=(k,h,0,0)$.
The parameter values for each transverse mode $\ell$ are given by $p_0 + \lambda_\ell \overline{p}$ and are indicated by white points in Fig.~\ref{Fig:space_P}. $\lambda_\ell$ are exemplary eigenvalues of the Laplacian matrix $\mathbf{L}$ \eqref{eq:L}, which correspond to the transverse modes. 

The points on the green dashed line of Fig.~\ref{Fig:space_P} correspond to the spectrum structure $\text{0}_3\text{I}_1\text{II}_0\text{U}_0$, i.e., all transverse modes are stable, but the mode with $\lambda_0=0$ is unstable. Hence, the agent formation is stable, while the trajectory can deviate from the virtual leader trajectory for large delay. 
As the delay becomes smaller ($\tau=2$) in Fig.~\ref{Fig:space_P}(b), all points including $p_0$ fall into the cross-hatched region, and therefore the formation configuration is stable and follows the virtual leader stably. 
Similarly, the blue dashed line illustrates how the transverse stability changes as the delay increases or decreases. The red dashed line provides an example where the formation configuration cannot be stabilized by any delay or coupling.

To illustrate the effect of different coupling formation configurations, we compute the bifurcation diagram in $\lambda$ and $h_0$ parameter space, see Fig.~\ref{Fig:Stable_Region_Coupled}. 
\begin{figure}
\centering	    
\includegraphics[width=5.2in,height=3in]{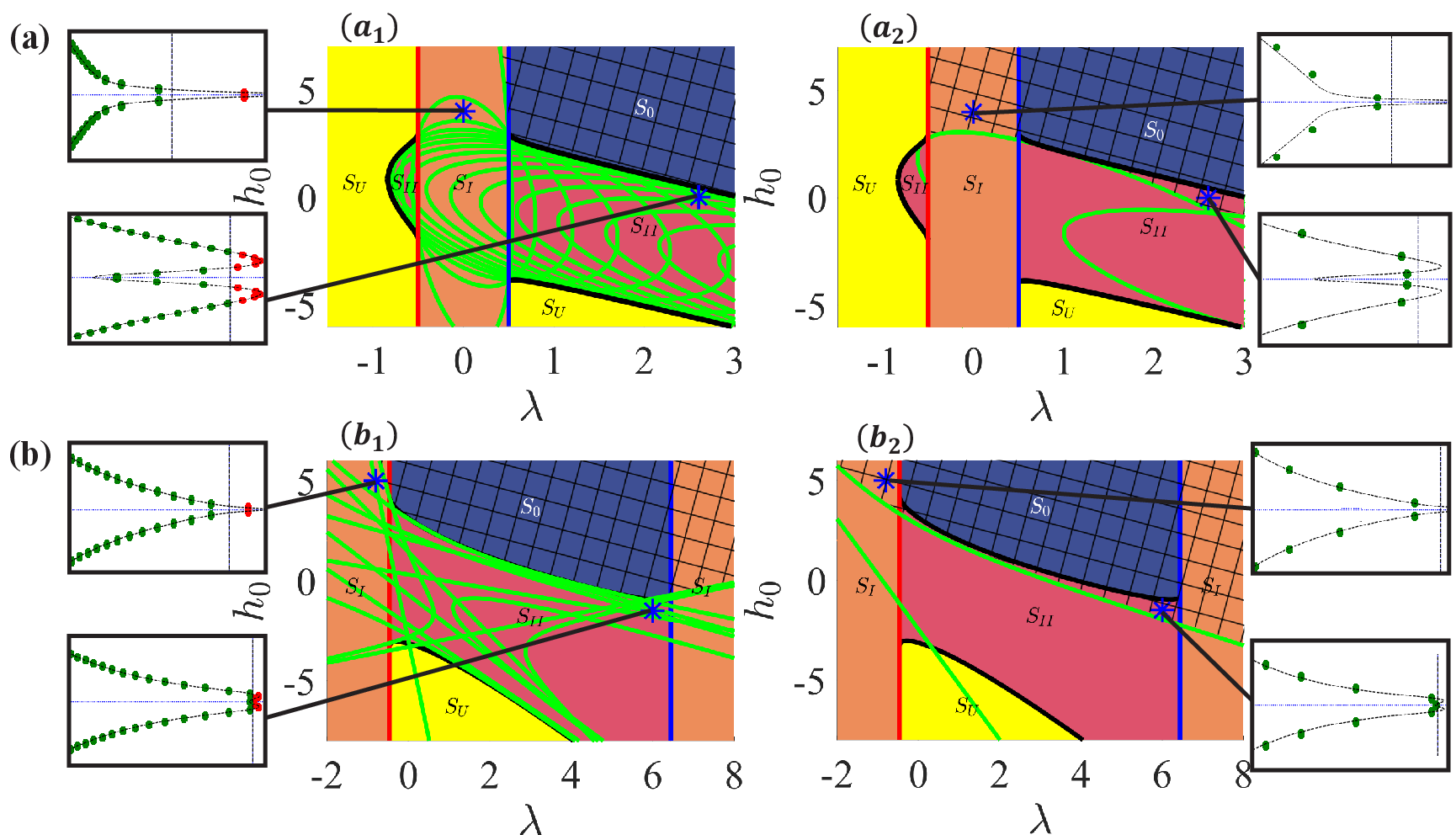}
    \caption{Bifurcation diagrams in the $(\lambda,h_0)$-plane are shown for different time delays:  
    $(a_1)~\tau=20$, $(a_2)~\tau=2$, 
    $ (b_1)~\tau=5$, and $(b_2)~\tau=1$. 
    The parameter values are fixed as follows: 
    (a) $p_0=(0,h_0,1.5,3)$, $\overline{p}=(3,1,0,0)$;  
    (b) $p_0=(3,h_0,1.5,3)$, $\overline{p}=(1.3,1,-2,0)$.
    The green lines indicate bifurcation curves. 
    Stability regions for fixed time delays are highlighted with the cross-hatched area.
    The sub-figures illustrate the spectra.
\label{Fig:Stable_Region_Coupled}}
\end{figure}
Since $\lambda$ are the eigenvalues of the Laplacian matrix $\mathbf{L}$, this takes into account different possible coupling formation configurations, i.e., $\lambda$ is the general expression for $\lambda_\ell$. 
Figures Fig.~\ref{Fig:Stable_Region_Coupled}(a) and (b) correspond to 
$k_0=0$ and $k_0=3$ respectively.
The stable active agent formations correspond to the cross-hatched area in the parameter space for fixed delays, and the blue $S_0$ region when stability is desired for arbitrary positive delays. 

The bifurcation lines create boundaries for the stability region, and are given by the explicit parametric expression
\begin{align}
&\lambda(\omega) = 
        \frac{\omega^2 - k_0 - h_0^\tau \omega \sin(\omega \tau) - k_0^\tau \cos(\omega \tau)}{k + h^\tau \omega \sin(\omega \tau) + k^\tau \cos(\omega \tau)},
\label{Eq:lambda_omega}\\
&h_0(\omega)=  
        \frac{\sin(\omega \tau)}{\omega} \left[ k_0^\tau  - \lambda(\omega) k^\tau  \right]   -  \left[ h_0^\tau +\lambda(\omega)h^\tau\right]\cos(\omega \tau)+ \lambda(\omega)h,
\label{Eq:h_0_omega}
\end{align}
which can be obtained from the following characteristic equation of the coupled system \eqref{Eq:eq:Gener_MSF-1-again}:
\begin{equation}
\label{eq:chareq_coupled}
         \det \left[ \mu I - (M-\lambda P) - (M^{\tau}-\lambda P^{\tau})  e^{-\mu\tau} \right] = 0,
\end{equation}
by substituting $\mu=i\omega$ and solving it with respect to the real eigenvalue $\lambda$ (the coupling is symmetric) and $h_0$.

Similarly to the bifurcation diagram in Fig.~\ref{Fig:Stable_Region}, the stability region becomes larger for small delays and shrinks to $S_0$ for large delays.
As a result, the system can also be stabilized in $S_I$ and $S_{II}$ for smaller delays.

\subsection{Stability of formation of non-symmetrically coupled agents}\label{Secsub_ContourMap}

We remind that stability of the active agent formation is described by system \eqref{Eq:eq:Gener_MSF-1-again}, where $\lambda_\ell$ are eigenvalues of the Laplacian matrix $\mathbf{L}$. 
The corresponding characteristic equation is \eqref{eq:chareq_coupled}.

For non-symmetrically coupled agents, $\lambda_{\ell}$ are generally complex.
Therefore, it is convenient to represent the stability region in the complex plane for $\lambda\in \mathbb{C}$. 
From the characteristic equation \eqref{eq:chareq_coupled}, we find
\begin{equation}
\label{eq:Characteristic_lambda_mu}
        \mu^2+h_0\mu+k_0+\left(h^\tau_0\mu+k^\tau_0\right)e^{-\mu\tau}+\lambda\left[\mu h+k+\left(\mu h^\tau+k^\tau\right)e^{-\mu\tau}\right]=0.
\end{equation}
Then, the explicit expression for $\lambda$ is
\begin{equation}
\label{Eq:eq:lambd_mu}
\lambda(\mu)=
        \frac{-\mu^2-\mu h_0-k_0-\left(\mu h^{\tau}_0+k_0^{\tau}\right)e^{-\mu\tau}}{\mu h +k+\left(k^{\tau}+\mu h^{\tau}\right) e^{-\mu\tau}},
\end{equation}
which can be used to find the bifurcation curves and the stability region in the $\lambda$ complex plane parametrically. 
Specifically, when the eigenvalues $\mu$ are purely imaginary, this corresponds to $\mu=i\omega$, $\omega\in\mathbb{R}$, and the corresponding curves in the parameter space are given parametrically as $\lambda(i\omega)$, $\omega\in\mathbb{R}$. 

By plotting the bifurcation curves as outlined above, we obtain the boundary of the stability region, which is illustrated by the green lines in Figs.~\ref{Fig:MSF_curve}(a) and (b) for two sets of parameters.
\begin{figure}
\centering	
  \subfigure{\includegraphics[width=6.3cm]{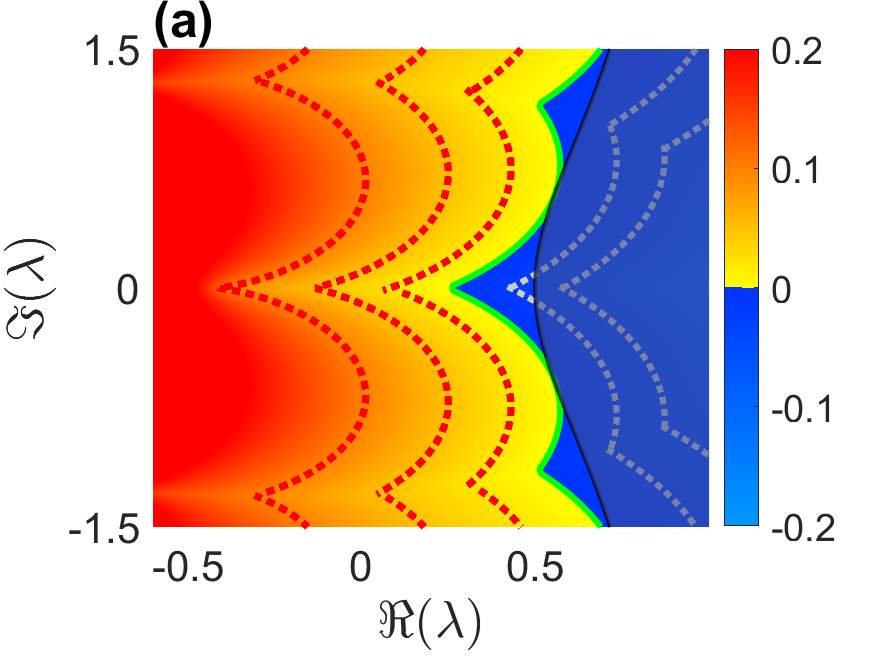}}
  \subfigure{\includegraphics[width=6.3cm]{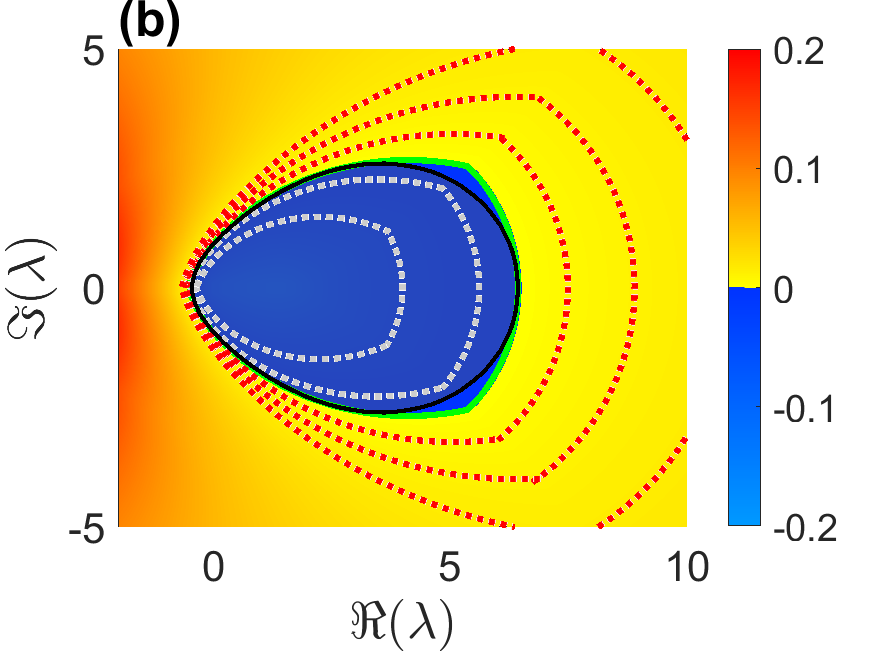}}
  \subfigure{\includegraphics[width=6.0cm]{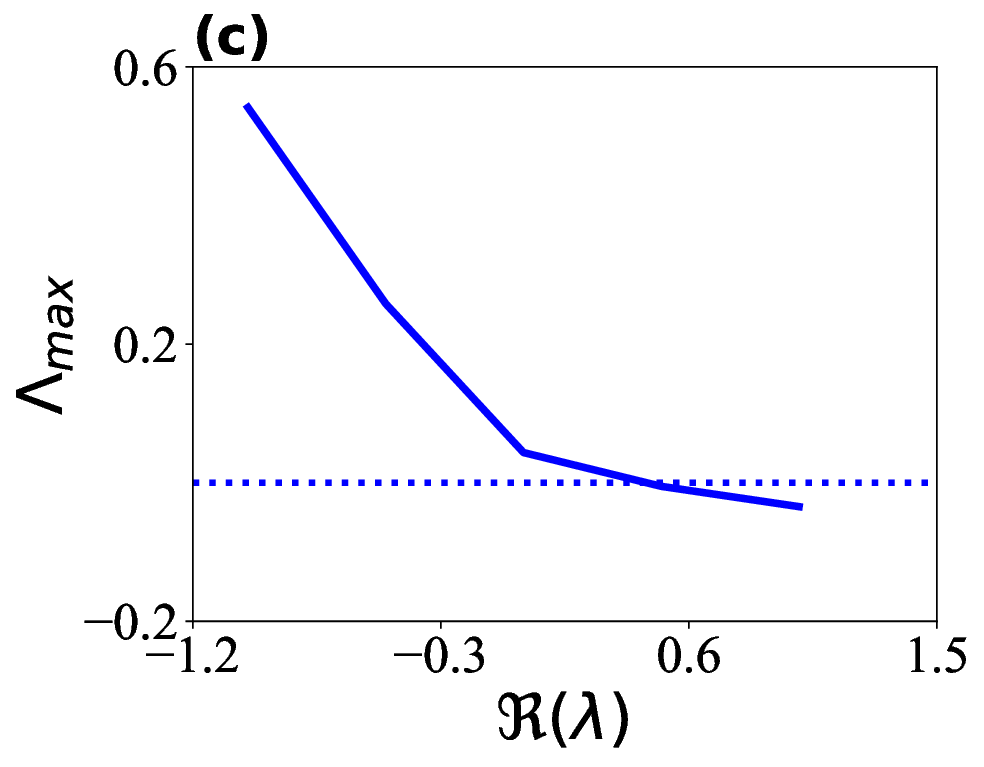}}
  \subfigure{\includegraphics[width=6.0cm]{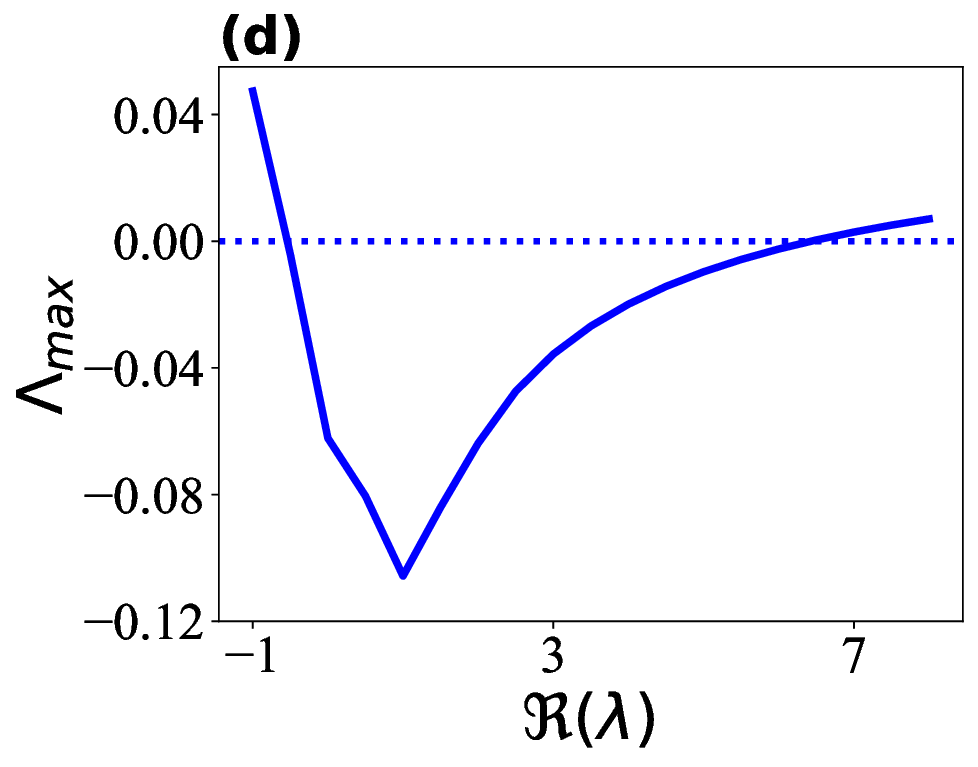}}
    \caption{ Master stability function $\Lambda_{\max}=\max ( \mathrm{Re}\, \mu )$ for delay-coupled system  \eqref{Eq:eq:Gener_MSF-1-again}. 
    (a) and (c) with fixed parameters $p_0=(0,6,1.5,3)$, $\overline{p}=(3,1,0,0)$;
    (b) and (d) with fixed parameters  $p_0=(3,6,1.5,3)$, $\overline{p}=(1.3,1,-2,0)$. 
    For (a) and (b), regions belonging to absolute stability are colored dark blue, the lines in the figures represent contour lines.
    The curve where $\Lambda_{\max}=0$ is given as green lines.  
    The red and gray dashed lines correspond to $\Lambda_{\max}>0$ and $\Lambda_{\max}<0$, respectively. 
    (c) and (d) show how the largest Lyapunov exponent $\Lambda_{\max}$ varies with $\Re(\lambda)$ and fixed $\Im(\lambda)=0$.
\label{Fig:MSF_curve} }
\end{figure}
Additionally, Figs.~\ref{Fig:MSF_curve}(a) and (b) show the 
largest Lyapunov exponent, defined as $\Lambda_{\max}=\max_j ( \Re (\mu_j) )$, as a color map. 
This is computed numerically directly from equation \eqref{eq:chareq_coupled}. 
It is clear that the stability boundary $\Lambda_{\max}=0 $ corresponds to the analytical green bifurcation line.
Furthermore, the dark blue colored areas correspond to regions of absolute stability. 
Different parameters lead to different shapes of the stable regions, as can be seen in Figs.~\ref{Fig:MSF_curve}(a) and (b).
Furthermore, the red and gray dashed lines in the figures show the contour lines for $\Lambda_{\max}$, which are obtained by setting $\mu = \text{const} + i\omega$ into Eq.~\eqref{Eq:eq:lambd_mu} and plotting $\lambda(\text{const} + i\omega)$. 
Detailed parameterized curves for the contour lines on the contour map are shown in \cref{Appendix_A}.

Figures~\ref{Fig:MSF_curve}(c) and (d) show how the largest Lyapunov exponent $\Lambda_{\max}$ varies with $\Re(\lambda)$, where  $\Im(\lambda)=0$ is fixed.

In the remaining part of this section, we will consider the case of $M^\tau = P = 0$ and a large delay limit.
This corresponds to a scenario in which the delay terms occur exclusively in the interaction, and the local feedback is non-delayed ($k=h=k_0^\tau=h_0^\tau=0$). 
The properties of the MSF in this case are reported in \cite{flunkert2010synchronizing}, and a notable feature is that the stability region asymptotically becomes a circle around the origin in the $\lambda$-plane as the delay increases. 
We demonstrate this for our system when $M^\tau = P = 0$, providing the first rigorous proof of this property. 
Moreover, we calculate the first-order correction terms that cause the bifurcation curve to deviate from a circular shape. 
To the first approximation, it has a rotated teardrop shape as in Fig.~\ref{Fig:MSF_differ_delay}. 
\begin{figure}
\centering	
\includegraphics[width=15cm]{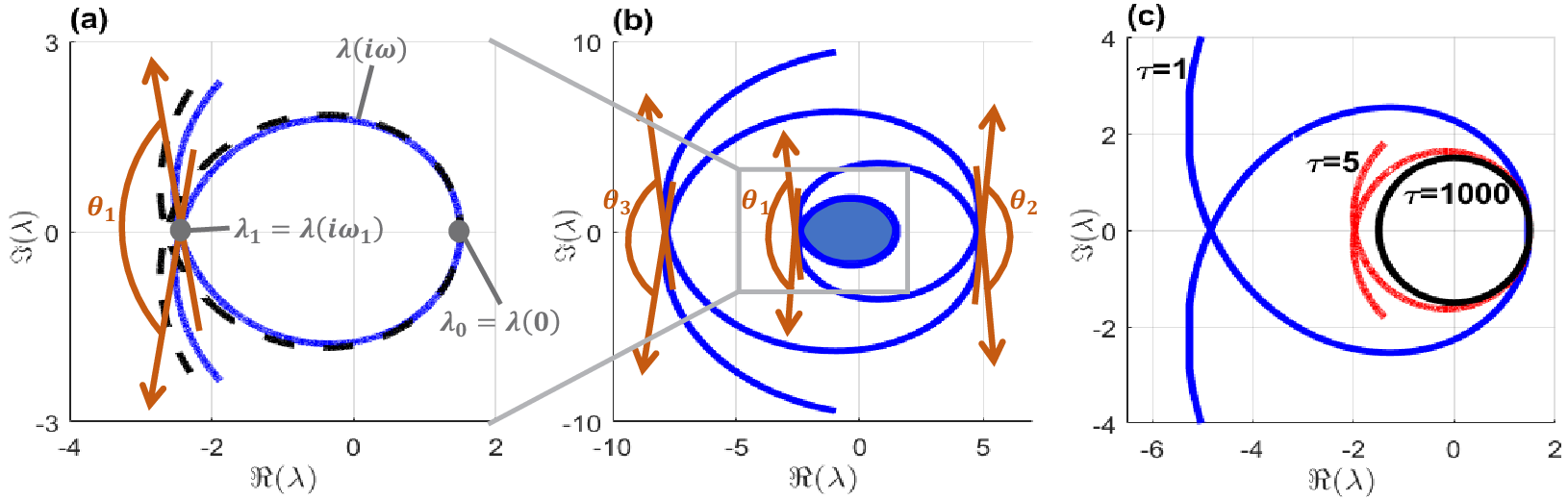}
    \caption{Illustration of the bifurcation curve in the $\lambda$-complex plane described by Theorem~\ref{Thm:Rotation_Symmetry}. 
    Panel (a) shows the bifurcation curve $\lambda(i\omega)$ with $\tau=3$, which is a partial zoom-in from panel (b), displaying details of the first self-intersection and the properties outlined in Theorem~\ref{Thm:Rotation_Symmetry}. 
    Panel (b) shows the bifurcation curve for $\tau=3$, illustrating the offset angles $\theta_j$ at self-intersection points. 
    Panel (c) shows bifurcation curves for various time delays: $\tau=1$ (blue), $\tau=5$ (red), $\tau=1000$ (black). 
    The dashed line in (a) corresponds to the asymptotic approximation curves given by Eq.~\eqref{Eq:Lambda_beta}. 
    The solid lines represent the exact  bifurcation curve described by Eq.~\eqref{Eq:eq:lambd_mu}. 
    The arrowed lines show the tangential line and its direction at the self-intersection point. 
    The parameters are fixed as  $p_0=(3,6,0,0)$ and $\overline{p}=(0,0,-2,0)$.  
\label{Fig:MSF_differ_delay}}
\end{figure}

\begin{theorem}
\label{Thm:Rotation_Symmetry}
  Let $M^\tau = P = 0$, $k_0 > 0$, $h_0>0$, and $k^\tau\ne 0$. Then the bifurcation curve $\lambda(i\omega)$ defined by Eq.~\eqref{Eq:eq:lambd_mu} has the following properties for sufficiently large $\tau$ and $|\omega\tau|<T$ with some $T>0$ independent on $\tau$.
  
      \textnormal{{(i)}}
        The following asymptotic representation holds
        \begin{align}
        \label{Eq:Lambda_beta}
        \lambda\left(i\omega\right)=
            \lambda_0e^{i\omega\tau} + i\omega \frac{k_0}{k^\tau}\left(\frac{h^\tau}{k^\tau}-\frac{h_0}{k_0}\right)e^{i\omega\tau}+ \mathcal{O}\left(\frac{1}{\tau^2}\right),
        \end{align}
        where $\lambda_0=\lambda(0)=-k_0/k^\tau$.
    
    \textnormal{{(ii)}}
        The curve $\lambda(i\omega)$  has self-intersection points at $\lambda(i \omega_j) = \lambda(-i \omega_j)\in \mathbb{R}$, where
        \begin{align*}
         \omega_j = 
            \frac{\pi j}{\tau}\left(1 + \frac{1}{\tau} \left(\frac{h^\tau}{k^\tau}-\frac{h_0}{k_0}\right)\right) + \mathcal{O}\left( \frac{1}{\tau^3}\right),\quad j\in\mathbb{Z}.   
        \end{align*}
        
    \textnormal{{(iii)}}
        The angles $\theta_j$ between tangential vectors of the curve at the self-intersection points $\lambda\left({i\omega_j} \right)$ are given by
        \begin{align}
        \label{Eq:Shift_angle}
        \theta_j=
            \pi+2j\pi\left(\frac{h^\tau}{k^\tau}-\frac{h_0}{k_0}\right)\frac{1}{\tau}+\mathcal{O}\left(\frac{1}{\tau^2}\right).
        \end{align}
        
    \textnormal{{(iv)}}
        The stability region of the MSF defined as $\Lambda(\lambda) := \max (\Re(\mu(\lambda)))$ with $\max$ taken over all characteristic roots $\mu$ of equation \eqref{eq:chareq_coupled} is confined to the connected region containing the origin. 

\end{theorem}
The geometric structure described by the theorem is illustrated in  Fig.~\ref{Fig:MSF_differ_delay}. 

\begin{proof}
Under the condition $M^{\tau}=P=0$, the characteristic equation \eqref{eq:chareq_coupled}  is reduced to 
\begin{align}
\label{eq:cheqtheorem}
        \det \left[ \mu I - M +\lambda P^{\tau}  e^{-\mu\tau}  \right] = 0.
\end{align}
For $\mu=0$, equation \eqref{eq:cheqtheorem} reads 
\begin{align}
\label{eq:cheqtheoremmu0}
    \det \left[  - M +\lambda P^{\tau} \right] = 0.
\end{align}
Let $\lambda_0$ and $\boldsymbol{\upsilon}_0 $ be the corresponding eigenvalue and the eigenvector so that 
\begin{align*}
    \left[ - M +\lambda_0 P^{\tau}   \right]\boldsymbol{\upsilon}_0 = D_0 \boldsymbol{\upsilon}_0 = 0.
\end{align*}
A direct calculation gives
\begin{align}
\label{eq:lambda0v0}
\lambda_{0}=
        -\frac{k_{0}}{k^{\tau}},\quad
\boldsymbol{\upsilon}_{0}=
        \left[
        \begin{array}{c}
            1\\0
        \end{array}
        \right].
\end{align}
Now consider the case $\mu=i\omega~(\omega\ne 0)$. Then condition \eqref{eq:cheqtheorem} is equivalent to
\begin{align}
\label{Eq:Character_lambda_gamma}
        \left[i\omega I -  M+\lambda_{\omega} P^{\tau} e^{-i\omega\tau}\right]\boldsymbol{\upsilon}_{\omega}= 0,
\end{align}
where  $\boldsymbol{\upsilon}_{\omega}$ is the right null eigenvector of matrix $D_{\omega}=i\omega I - M +\lambda_{\omega} P^{\tau}  e^{-i\omega\tau} $, and $\lambda_\omega=\lambda\left(i\omega\right)$.
We apply the following transformations
\begin{align}
\label{Eq:lambda_gamma}
\lambda_{\omega}:=
  \lambda\left(i\omega\right)=                
        \lambda_0e^{i\omega\tau}+\frac{1}{\tau}e^{i\omega\tau} \beta_{\omega},
\end{align}
\begin{align}
\label{Eq:v_gamma}
  \boldsymbol{\upsilon}_{\omega}:=
    \boldsymbol{\upsilon}\left(i\omega\right)=
        \boldsymbol{\upsilon}_{0}+\frac{1}{\tau}\tilde{\boldsymbol{\upsilon}}_{\omega},
\end{align}
where $\beta_{\omega}$ and $\tilde{\boldsymbol{\upsilon}}_{\omega}$ are new unknown function and a vector, which describes the perturbation terms for $\lambda_\omega$ and $\boldsymbol{\upsilon}_\omega$.
By substituting Eqs.~\eqref{Eq:lambda_gamma} and \eqref{Eq:v_gamma} into Eq.~\eqref{Eq:Character_lambda_gamma}, we obtain
\begin{align}
\label{Eq:lambda_gamma_2}
        \left[i\omega I - M +\lambda_0P^{\tau}  +\frac{1}{\tau}P^\tau\beta_\omega     \right]\left(\boldsymbol{\upsilon}_{0}+\frac{1}{\tau}\tilde{\boldsymbol{\upsilon}}_{\omega}\right) = 0.
\end{align}
Considering the leading order terms $\mathcal{O}(1/\tau)$ of \eqref{Eq:lambda_gamma_2}, we get
\begin{align}
\label{Eq:beta_gamma}
D_0\tilde{\boldsymbol{\upsilon}}_{\omega}=
        - (i\omega\tau I+\beta_\omega P^\tau) \boldsymbol{\upsilon}_0.
\end{align}
Equation~\eqref{Eq:beta_gamma} is solvable if and only if $\left(i\omega\tau I+\beta_\omega P^\tau\right)\boldsymbol{\upsilon}_0\in \text{Im}(D_0)$, or, equivalently
\begin{align}
\label{eq:ker}
        \left(i\omega\tau I+\beta_\omega P^\tau\right)\boldsymbol{\upsilon}_0\,\,\bot \,\ker\left(D_0^T\right).
\end{align}
The condition \eqref{eq:ker} can be rewritten using the left null eigenvector $\boldsymbol{\upsilon}^{+}_0 = \left[h_{0}-h^{\tau}\frac{k_{0}}{k^{\tau}} , 1\right]$ of the matrix $D_0$, which satisfies $\boldsymbol{\upsilon}^{+}_0 D_0=0$. 
As a result, the solvability condition for \eqref{Eq:beta_gamma} leads to the equation
\begin{align}
\label{eq:solvability1}
        \boldsymbol{\upsilon}^{+}_0 (i\omega\tau I+\beta_\omega P^\tau)\boldsymbol{\upsilon}_0=i\omega\tau \left(h_0-\frac{k_0h^\tau}{k^\tau}\right)+k^{\tau}\beta_{\omega}=0,
\end{align}
which can be solved with respect to $\beta_\omega$:
\begin{align}
\label{Eq:bets_gamma}
\beta_\omega=
        -i\omega\tau\frac{\boldsymbol{\upsilon}^{+}_0\boldsymbol{\upsilon}_0}{\boldsymbol{\upsilon}^{+}_0P^\tau\boldsymbol{\upsilon}_0}=
         -i\omega\tau\left(\frac{h_0k^\tau-k_0h^\tau}{(k^\tau)^2}\right).
\end{align}
By substituting \eqref{Eq:bets_gamma} into \eqref{Eq:lambda_gamma}, we yield \eqref{Eq:Lambda_beta}. This completes the proof of the statement (i) of the theorem.

Next, we determine the offset angle at which the bifurcation curve self-intersects, i.e., at the point where $\lambda(i\omega)=\overline{\lambda(i\omega)}$.
Due to $k=h=k_0^\tau=h_0^\tau=0$, Eq.~\eqref{Eq:eq:lambd_mu} simplifies to
\begin{align*}
\lambda\left({i\omega}\right)=
        \frac{\omega^2-i\omega h_0-k_0}{i\omega h^\tau+k^\tau}e^{i\omega\tau}.
\end{align*}
Imposing the self-intersection condition $\lambda(i\omega)=\overline{\lambda(i\omega)}$ yields the requirement that $\lambda(i\omega)$ is real, i.e., 
\begin{align*}
\Im[\lambda(i\omega)]=
        \frac{\left(-k_0k^\tau+\omega^2\left(k^\tau-h_0h^\tau\right)\right)\sin(\omega\tau)}{(\omega h^\tau)^2+(k^\tau)^2}
        -\frac{\omega \left(h_0 k^\tau+ h^\tau\left(\omega^2-k_0\right)\right)\cos(\omega\tau)}{(\omega h^\tau)^2+(k^\tau)^2}=0,
\end{align*}
from which we obtain
\begin{align*}  
\tan(\omega\tau)=
        \frac{\omega \left(h_0 k^\tau+h^\tau\left(\omega^2-k_0\right)\right)}{k^\tau\left(\omega^2-k_0\right)-\omega^2h_0h^\tau}.
\end{align*}
Defining $\Omega=\omega\tau$ and expanding in $1/\tau$ we arrive at
\begin{align}
\label{eq:tangent_Omega}
\tan \Omega=
        \frac{\Omega}{\tau}\left(\frac{h^\tau}{k^\tau}-\frac{h_0}{k_0}\right)+\mathcal{O}\left(\frac{1}{\tau^3}\right),
\end{align}
which allows us to approximate the solutions using an asymptotic expansion. 
Let $\Omega=\Omega_{0j}+\frac{1}{\tau}\tilde{\Omega}_j$, where $\Omega_{0j}=j\pi$ is the leading-order approximation. Substituting it into Eq.~\eqref{eq:tangent_Omega} gives
\begin{align*}
\tan\left(\Omega_{0j}+\frac{1}{\tau}\tilde{\Omega}_j\right)=
        \frac{\Omega_{0j}+\frac{1}{\tau}\tilde{\Omega}_j}{\tau}\left(\frac{h^\tau}{k^\tau}-\frac{h_0}{k_0}\right)+\mathcal{O}\left(\frac{1}{\tau^3}\right),~j\in\mathbb{Z}.
\end{align*}
Taking into account the periodicity of the tangent, we can solve for $\omega_j=\frac{\Omega}{\tau}$, yielding
\begin{align*}
\omega_j=
        \frac{j\pi}{\tau}\left[1+\frac{1}{\tau}\left(\frac{h^\tau}{k^\tau}-\frac{h_0}{k_0}\right)\right]+\mathcal{O}\left(\frac{1}{\tau^3}\right).
\end{align*}
The proof of the statement (ii) is complete.

We further compute the slope angle $\alpha_j$ of the bifurcation curve near the intersection point  $\omega_j$. 
This requires evaluating the derivative of $\lambda\left(i\omega\right)$ with respect to $\omega$ at $\omega_j$:
\begin{align*}
\left.\frac{\partial\lambda(i\omega)}{\partial\omega}\right|_{\omega=\omega_j}=&
        \left[\frac{2\left(\frac{\pi j}{\tau}\right)k^\tau-i\left(h_0k^\tau+k_0h^\tau\right)}{2i\left(\frac{\pi j}{\tau}\right) h^\tau k^\tau+\left(k^\tau\right)^2}\right](-1)^j
    \\&
        +\left[i\tau\frac{-i\left(\frac{\pi j}{\tau}\right)h_0-k_0}{i\left(\frac{\pi j}{\tau}\right)h^\tau+k^\tau}\right](-1)^j+\mathcal{O}\left(\frac{1}{\tau^2}\right)
    \\=&
        (-1)^j\left[i\frac{-k^\tau h_0+k_0 h^\tau}{\left(k^\tau\right)^2}+\frac{1}{\tau}\frac{2\pi j\left(\left(k^\tau\right)^2-k^\tau h^\tau+\left(h^\tau\right)^2k_0\right)}{\left(k^\tau\right)^3}\right.\\&\quad\quad\quad\quad\left.-i\tau\frac{k_0}{k^\tau}+\frac{\pi j\left(h_0k^\tau-k_0h^\tau\right)}{\left(k^\tau\right)^2}\right]+\mathcal{O}\left(\frac{1}{\tau^2}\right).
\end{align*}
This derivative has the following form:
\begin{align*}
\left.\frac{\partial\lambda(i\omega)}{\partial\omega}\right|_{\omega=\omega_j}=
        (-1)^j\left[i\left(A_{-1}\tau+A_0\right)+B_0+\frac{B_1}{\tau}\right]+\mathcal{O}\left(\frac{1}{\tau^2}\right),
\end{align*}
where  
\begin{align*}
    &A_{-1}=
        -\frac{k_0}{k^\tau},\quad
    A_0=
        \frac{-k^\tau h_0+k_0 h^\tau}{\left(k^\tau\right)^2}; \\
    &B_0=
        \frac{\pi j\left(h_0k^\tau-k_0h^\tau\right)}{\left(k^\tau\right)^2},\quad
    B_1=
        \frac{2\pi j\left(\left(k^\tau\right)^2-k^\tau h^\tau+\left(h^\tau\right)^2k_0\right)}{\left(k^\tau\right)^3}.
\end{align*}
Then, the angle $\alpha_j$ is given by
\begin{align*}
\cot(\alpha_j)=&
        \frac{B_0}{A_{-1}\tau}+\mathcal{O}\left(\frac{1}{\tau^2}\right).
\end{align*}
Substituting the expressions for $A_{-1}$ and $B_0$, we find
\begin{align*}
\cot(\alpha_j)=
        j\pi \left(\frac{h^\tau}{k^\tau}-\frac{h_0}{k_0}\right)\frac{1}{\tau}+\mathcal{O}\left(\frac{1}{\tau^2}\right),
\end{align*}
and hence,
\begin{align*}
\alpha_j=
        \frac{\pi}{2}+j\pi \left(\frac{h^\tau}{k^\tau}-\frac{h_0}{k_0}\right)\frac{1}{\tau}+\mathcal{O}\left(\frac{1}{\tau^2}\right).
\end{align*}
Because the bifurcation curve is symmetric with respect to the imaginary axis, the offset angle $\theta_j$ at the self-intersection points is given by
\begin{align*}
\theta_j=
        2\alpha_j=\pi+2j\pi \left(\frac{h^\tau}{k^\tau}-\frac{h_0}{k_0}\right)\frac{1}{\tau}+\mathcal{O}\left(\frac{1}{\tau^2}\right),~j\in\mathbb{Z}.
\end{align*}
This completes the proof of the statement (iii).

The bifurcation curves $\lambda(i\omega)$ determine the stability boundary where $\Lambda_{\max} = 0$. 
Properties (i))--(iii) imply that these curves encircle the origin. 
Furthermore the spectrum of all characteristic roots $\mu$ for $\lambda = 0$ is equal to the spectrum of $M$ and it is stable under the assumptions of the theorem. 
Therefore, the stable region is necessarily confined to the connected component containing the origin. 
This confirms that statement (iv) holds.
\end{proof}
\begin{remark}
According to Theorem~\ref{Thm:Rotation_Symmetry}, the stability region of the MSF is asymptotically circular for large-delay coupling. 
While this general result was announced in \cite{flunkert2010synchronizing}, a rigorous proof is given here for the first time, albeit for a particular class of delay-coupled systems.
\end{remark}

Figure~\ref{Fig:MSF_differ_delay} illustrates the key properties of the bifurcation curve $\lambda(i\omega)$ as discussed in  Theorem~\ref{Thm:Rotation_Symmetry} in the complex $\lambda$-plane for various time delays. 
Figures~\ref{Fig:MSF_differ_delay}(a) and~\ref{Fig:MSF_differ_delay}(b) focus on the case $\tau = 3$, (a) zooms in on the stable region and the first self-intersection in (b). 
Figure~\ref{Fig:MSF_differ_delay}(a) highlights the curve Eq.~\eqref{Eq:eq:lambd_mu}  (solid line),  the asymptotic curve Eq.~\eqref{Eq:Lambda_beta} (dashed line),  the starting point $\lambda_0$, the first self-intersection point $\lambda_1$, tangents (arrowed line), and offset angle $\theta_1$, as calculated from Theorem~\ref{Thm:Rotation_Symmetry}.
Figure~\ref{Fig:MSF_differ_delay}(b) displays a larger piece of the bifurcation curve when $\tau=3$, emphasizing multiple self-intersection points $\lambda_j$ and the corresponding tangent directions and offset angles $\theta_j$. 
Figure \ref{Fig:MSF_differ_delay}(c) shows the bifurcation curves for different time delays $\tau$. 
For $\tau = 1$, the curve forms a teardrop shape and has no rotational symmetry. 
As $\tau$ increases, for instance, at $\tau = 5$, the curve starts to become circular. 
When $\tau\to\infty$, the curve converges to a circle around the origin.

\subsection{Numerical simulation examples with three agents}\label{Subsec_Numer_Coupled}

To demonstrate how coupling can stabilize an otherwise unstable agent formation, we performed numerical simulations of a three-agent coupled formation system.
Figure~\ref{Fig:Coupled_Motion} presents the motion trajectories of agents, the time evolution of the motion error, and the asymptotic continuous spectrum of the error system.  
\begin{figure}
    \centering
\includegraphics[width=\textwidth]{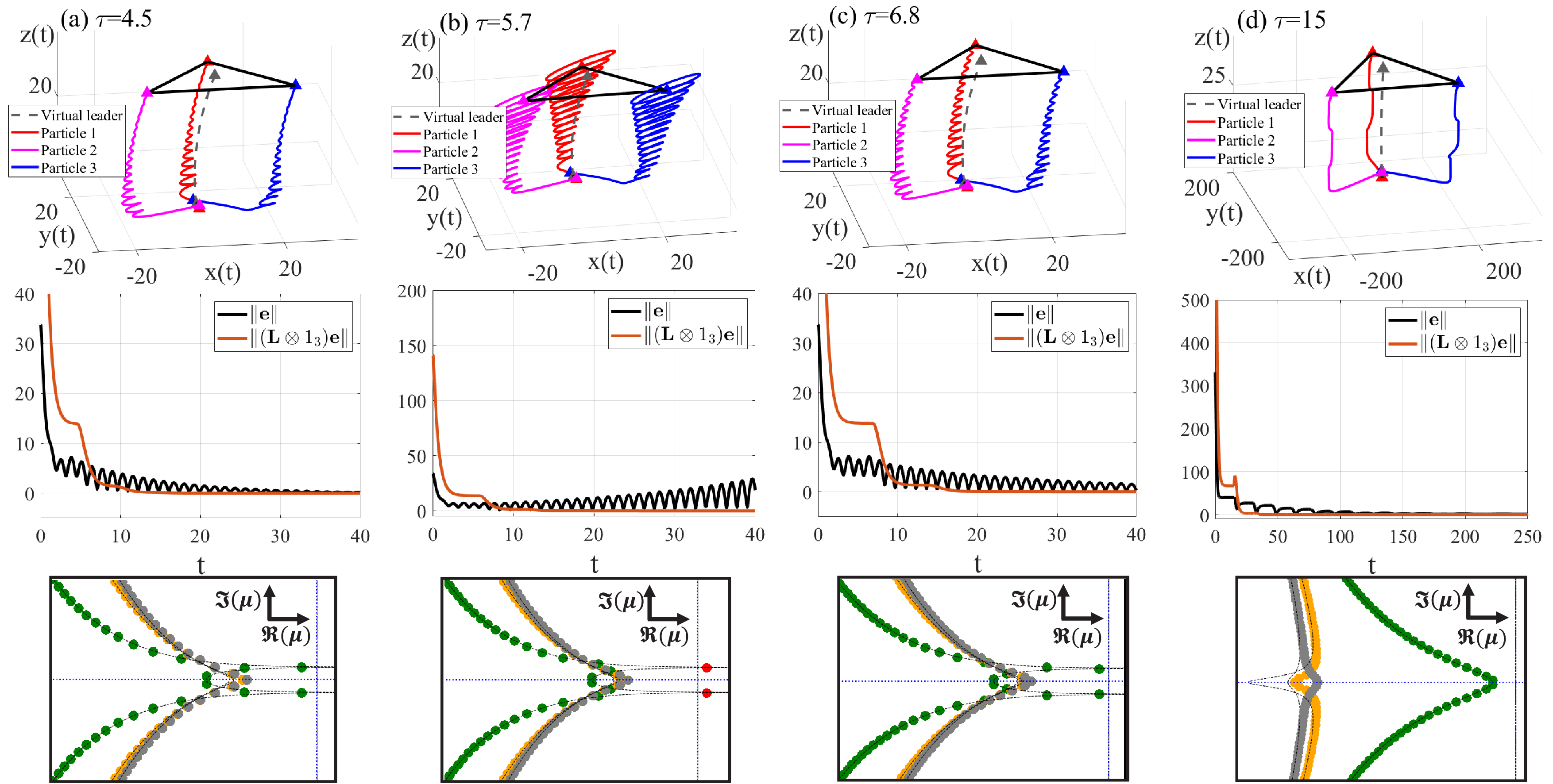}
    \caption{Illustration of the active agents formation with coupling between agents described by the Laplacian matrix \eqref{Eq:L_coupling}. 
    Time delay is fixed: (a) $\tau=4.5$, (b) $\tau=5.7$, (c) $\tau=6.8$, and (d) $\tau=15$. 
    The parameters are fixed as follows: (a)--(c) $k_0=6,~h_0=0,~k_0^\tau=0.3,~h_0^\tau=0$ and $k=3,~h=3,~k^\tau=-0.5,~h^\tau= 0$; (d)  $k_0=2,~h_0=3,~k_0^\tau=1.5,~h_0^\tau=1.2$ and $k=3,~h=3,~k^\tau=-0.5,~h^\tau=0$.  
    \textbf{Top row}: 
        the motion trajectories (red, pink, and blue curves) of the agents (solid triangles), which are given by Eqs.~\eqref{eq:Motion_AP_a}--\eqref{eq:Motion_AP_b} with the control input \eqref{eq:ControlInput2};  the gray dashed lines denote the trajectories (see Eq.~\eqref{Eq:target_trajectory}) of the virtual leader.
    \textbf{Middle row}: 
        the system errors, including the tracking error $\|\mathbf{e}\|$  (black lines, as defined in Eqs.~\eqref{eq:DDE-e}--\eqref{eq:DDE-xi}), and the formation error (orange lines) $\|(\mathbf{L}\otimes 1_3)\mathbf{e}\|$.
    \textbf{Bottom row}: 
        Spectrum of the error system, containing three branches for eigenvalues $\lambda=4$ (yellow) $\lambda=0$ (green and red) $\lambda=5$ (gray) of the Laplacian matrix \eqref{Eq:L_coupling}. 
\label{Fig:Coupled_Motion}}
\end{figure}
The motion of agents are governed by Eqs.~\eqref{eq:Motion_AP_a}--\eqref{eq:Motion_AP_b} with the control input \eqref{eq:ControlInput2}. 
The error system is defined in Eqs.~\eqref{eq:DDE-e}--\eqref{eq:DDE-xi}, which is evaluated with respect to the target trajectory. 
The target trajectory, the parameters $k_0,h_0,k^\tau_0,h^\tau_0$, and time delays are the same as in Fig.~\ref{Fig:Motion}. 
The coupling parameters in Eq.~\eqref{eq:ControlInput2}  are fixed to $k=3,~h=3,~k^\tau=-0.5,~\text{and}~h^\tau=0$. 
The coupling structure is determined by the following Laplacian matrix:
\begin{align}
\label{Eq:L_coupling}
\mathbf{L}=
    \left[\begin{array}{ccc}
        3 & -2 &-1  \\
        -2 &3 &-1\\
        -2&-1&3
    \end{array}\right].
\end{align} 
The asymptotic continuous spectrum of the error system is described by \eqref{eq:ACS-complex}.
    
The simulations followed a setup similar to the uncoupled case (see Fig.~\ref{Fig:Motion}). 
This allows for a direct comparison of agent formation behavior before and after the introduction of coupling.
We observe that a suitable coupling improves the formation stability in two key aspects:
(i) the agents converge more reliably to the desired formation configuration;
(ii) the amplitude of delay-induced oscillations in their trajectories is significantly reduced.

In Fig.~\ref{Fig:Coupled_Motion}, we show the stabilizing effect of inter-agent coupling, which significantly enhances the resilience to a communication delay of agents.
Compared to Fig.~\ref{Fig:Motion}, coupling can reduce delay-induced oscillations and form a more robust formation configuration.
After introducing the coupling, the trajectories of agents remain following the virtual leader when delays are $\tau = 4.5, 6.8$, and $15$, see Figs.~\ref{Fig:Coupled_Motion}(a), (c), and (d). 
And agents can achieve the desired formation configuration for all delays, see Figs.~\ref{Fig:Coupled_Motion}(a)--(d).   
The middle row confirms the above coordinated behavior through the bounded decay of errors. 
The tracking error $\|\mathbf{e}(t)\|$ (black lines) has a significant reduction compared to Fig.~\ref{Fig:Motion}, i.e., the oscillations caused by the delay are significantly attenuated.
The formation error $\|(\mathbf{L} \otimes 1_3)\mathbf{e}(t)\|$ (shown by the orange lines) represents the difference in the relative positions of the active agents. When this error asymptotically converges to $0$, it indicates that the desired formation shape can be achieved, regardless of the delays.  
The bottom row shows the spectrum plots of the error system, reflecting the stability corresponding to the eigenvalues $\lambda=4,0,5$ of the coupling matrix $\mathbf{L}$ \eqref{Eq:L_coupling}. 

\section{Conclusions}\label{sec:Conclusions}
In summary, this study provides a comprehensive analysis of how time-delayed interactions affect the stability of motion and formation in active agent systems.
We consider a general linear model incorporating delay effects into the dynamics, guided by a virtual leader. 
To quantify deviations in position and velocity of motion, we introduce error variables, which result in a high-dimensional linear DDE model with inertial effects describing the error dynamics. 
This equation serves as the basis for our stability analysis. 

Our investigation focused on two fundamental scenarios: (1) agent motion driven only by a virtual leader (uncoupled case), and (2) agent motion influenced both by  a virtual leader and mutual interactions among agents (coupled case).  
Furthermore, we applied recent results from \cite{Yanchuk2022b,wang2024universal} to provide explicit parametric conditions for the formation to be stable for all delays (we called it absolutely stable), unstable for all delays, or possess explicitly defined stability domains. For this, we used the classification of the spectra of linear DDEs from \cite{wang2024universal}.
All analytical findings were corroborated by numerical simulations, including a representative example of a pattern formation involving three active agents.

In the coupled case,  we employ the MSF approach to analyze the formation stability of active agents, which is a powerful tool for studying the stability of coupled error systems by decoupling the dynamics into distinct coupling modes.
For the case of complex Laplacian eigenvalues $\lambda$, we assess the stability of agent formation by mapping the largest Lyapunov exponent and identifying stability regions in the $(\Re(\lambda),\Im(\lambda))$-plane, as shown in Fig.~\ref{Fig:MSF_curve}. 
An interesting finding is that, as the interacting delay increases, the stability region gradually forms a circular shape centered at the origin. 
This work provides the first rigorous proof of this circular property, extending the results from \cite{flunkert2010synchronizing}. For finite delays, it has a teardrop shape, the properties of which are described in Theorem~\ref{Thm:Rotation_Symmetry}. 

A key feature of our results is that our theorems provide explicit conditions despite the large number of parameters involved in the control setup. These parameters include the elements of the control matrices $M$, $M^\tau$, $P$ and $P^\tau$, as well as the potentially arbitrary Laplacian matrix $\mathbf{L}$, which describes the interactions.

Our results could be useful for applications involving multiple agents interacting with time delays, such as in robotics, autonomous systems and networked control.
A possible area for future research would be to investigate how adaptive interactions or time-varying delays affect the dynamical system presented here.

\appendix
\section{Parameterized curves for contour lines of the master stability function}\label{Appendix_A}
In Sec.~\ref{Secsub_ContourMap} of the main text, we introduced the contour map of the MSF and identified the stable and unstable regions separated by the bifurcation curve. 
In the following, we provide more detailed parameterized curves for contour lines on the contour map (see Figs.~\ref{Fig:MSF_curve_Appendix}(a)  and (b)) and determine the main stable and unstable contour lines (see Figs.~\ref{Fig:MSF_curve_Appendix}(c) and (d)).
\begin{figure}
\centering	
  \subfigure{\includegraphics[width=6cm]{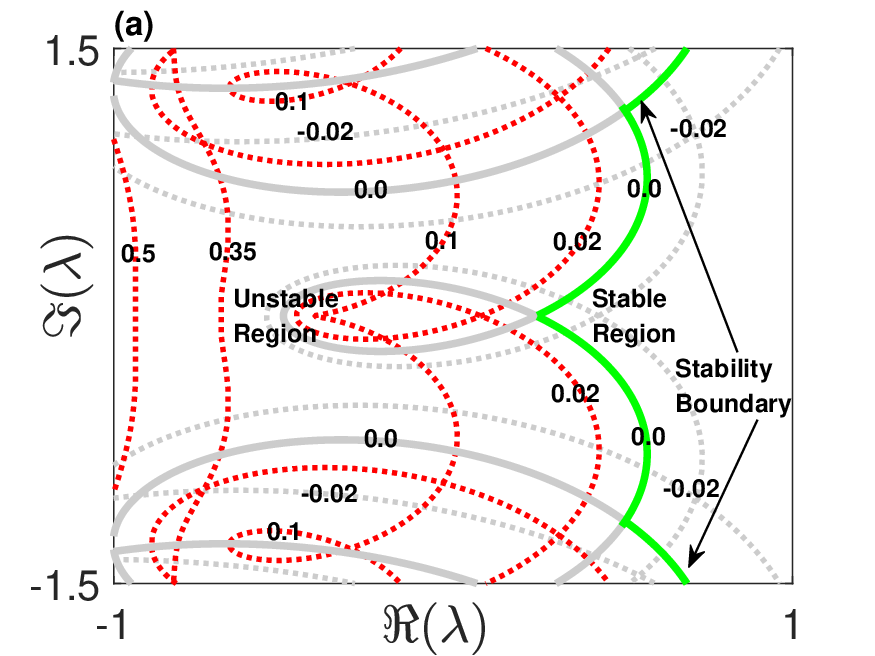}}
  \subfigure{\includegraphics[width=6cm]{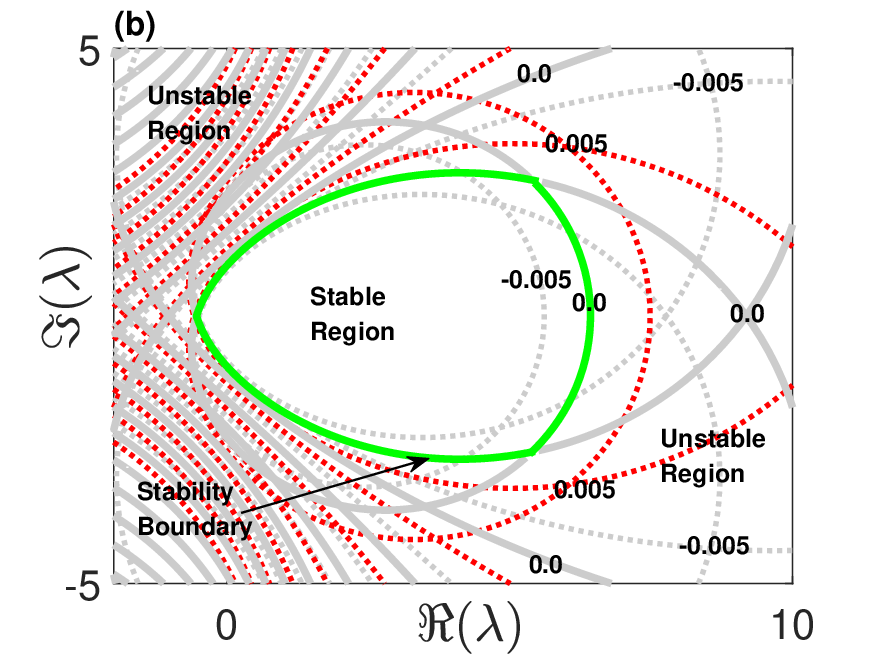}}
  \subfigure{\includegraphics[width=6cm]{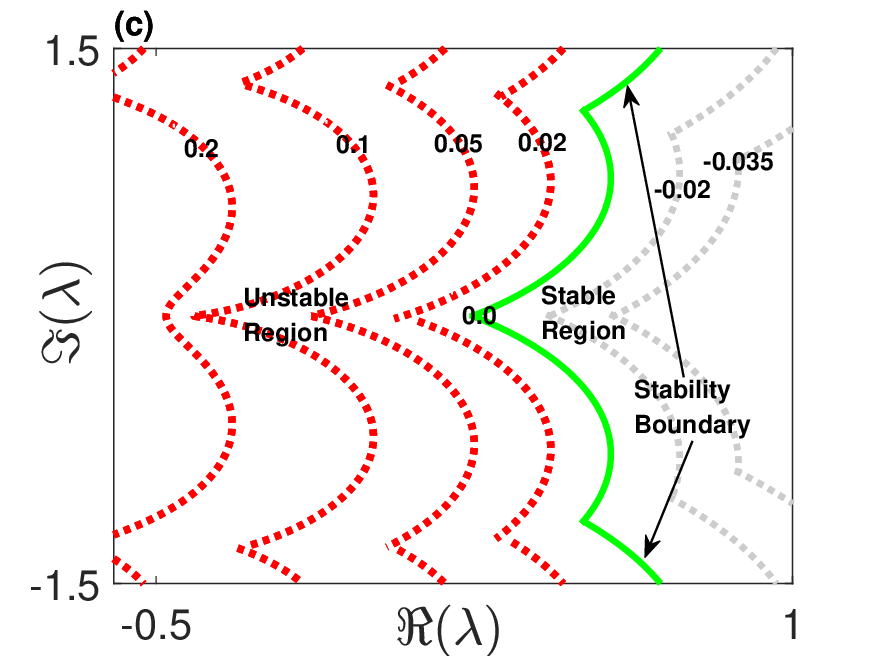}}
  \subfigure{\includegraphics[width=6cm]{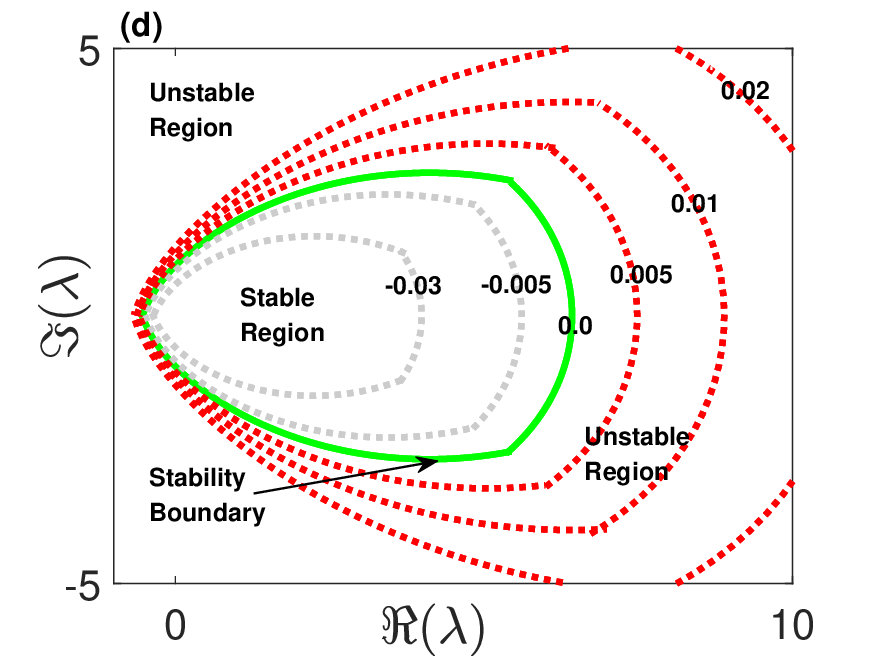}}
    \caption{Contour lines of the MSF, i.e.,  $\Lambda_{\max}=\max ( \mathrm{Re}\, \mu )$ for delay-coupled system  \eqref{Eq:eq:Gener_MSF-1-again} with $\tau=10$ and fixed parameters: 
    (a) and (c) $p_0=(0,6,1.5,3)$, $\overline{p}=(3,1,0,0)$;
    (b) and (d) $p_0=(3,6,1.5,3)$, $\overline{p}=(1.3,1,-2,0)$. 
    The curve where $\Lambda_{\max}=0$ is combined green and gray solid lines, where the green solid lines indicate the stability boundary. 
    The gray and red dashed lines correspond to the stable and unstable states, respectively.
\label{Fig:MSF_curve_Appendix} }
\end{figure} 
 



\bibliographystyle{siamplain}
\bibliography{references}
\end{document}